%% file: paper.tex
\documentclass[11pt]{article}

\usepackage{graphicx} % more modern
\usepackage{subfigure}
\usepackage{fullpage}
\usepackage[font=small,labelfont=bf]{caption}

\usepackage{algorithm}
\usepackage{algorithmic}

\usepackage[colorlinks=true,linkcolor=blue,citecolor=blue,breaklinks]{hyperref}
\usepackage[hyphenbreaks]{breakurl}
\usepackage[tracking]{microtype} % http://www.khirevich.com/latex/microtype/

\usepackage{amsthm}
\usepackage{amsmath}
\usepackage{amssymb}
\usepackage{mathtools}
\usepackage{mathrsfs}
\usepackage{relsize}
\usepackage{color}

\input{commands}

\definecolor{darkblue}{rgb}{0, 0.141, 0.758}
\definecolor{orange}{RGB}{213,114,0}

\title{Breaking Locality Accelerates Block Gauss-Seidel}

\author{
Stephen Tu \quad
Shivaram Venkataraman \quad
Ashia C. Wilson \\
Alex Gittens \quad
Michael I. Jordan \quad
Benjamin Recht
}

\begin{document}
\date{}

\maketitle

%\vskip 0.3in

\input{abstract}
\input{intro}
\input{background}
\input{results}
\input{related}
\input{experiments}
\input{conclusion}

\section*{Acknowledgements}
We thank Ross Boczar for assisting us with Mathematica support for non-commutative algebras,
Orianna DeMasi for providing useful feedback on earlier drafts of this manuscript,
and the anonymous reviewers for their helpful feedback.
ACW is supported by an NSF Graduate Research Fellowship. BR is generously supported by ONR awards N00014-11-1-0723 and N00014-13-1-0129, NSF award
CCF-1359814, the DARPA Fundamental Limits of Learning (Fun LoL) Program, a Sloan Research
Fellowship, and a Google Research Award.
This research is supported in part by DHS Award HSHQDC-16-3-00083, NSF CISE Expeditions Award
CCF-1139158, DOE Award SN10040 DE-SC0012463, and DARPA XData Award FA8750-12-2-0331, and gifts from
Amazon Web Services, Google, IBM, SAP, The Thomas and Stacey Siebel Foundation, Apple Inc., Arimo,
Blue Goji, Bosch, Cisco, Cray, Cloudera, Ericsson, Facebook, Fujitsu, HP, Huawei, Intel, Microsoft,
Mitre, Pivotal, Samsung, Schlumberger, Splunk, State Farm and VMware.

{
\small
\bibliography{paper}
\bibliographystyle{abbrv}
}

\clearpage
\appendix
\onecolumn
\renewcommand{\thesection}{A.\arabic{section}}

\input{notation}
\input{proofs_separation}
\input{proofs_convergence}
\input{proofs_random_coordinates}

%\bigskip
%\bibliographysup{paper}
%\bibliographystylesup{abbrv}

\end{document}

%% file: commands.tex
\newtheorem{theorem}{Theorem}[section]
\newtheorem{lemma}[theorem]{Lemma}

\newtheorem{proposition}[theorem]{Proposition}

\newcommand{\diag}{\mathrm{diag}}
\newcommand{\rank}{\mathrm{rank}}

\newcommand{\Kern}{\mathrm{Kern}}
\newcommand{\Range}{\mathcal{R}}

\newcommand{\E}{\mathbb{E}}
\newcommand{\R}{\ensuremath{\mathbb{R}}}
\newcommand{\C}{\ensuremath{\mathbb{C}}}
\newcommand{\ind}{\mathbf{1}}
\newcommand{\ip}[2]{\ensuremath{\langle #1, #2 \rangle}}
\newcommand{\twonorm}[1]{\lVert #1 \rVert_{2}}
\newcommand{\norm}[1]{\lVert #1 \rVert}

\newcommand{\abs}[1]{\ensuremath{| #1 |}}

\newcommand{\Tr}{\mathbf{Tr}}

\newcommand{\T}{\mathsf{T}}
\renewcommand{\Pr}{\mathbb{P}}

\newcommand{\eqrefstackrel}[2]{\overset{\mathclap{\eqref{#1}}}{#2}}
\newcommand{\bigkappa}{\mathlarger{\kappa}}

%% file: abstract.tex
\begin{abstract}

Recent work by Nesterov and Stich~\cite{nesterov16} showed that momentum
can be used to accelerate the rate of convergence for block Gauss-Seidel in
the setting where a fixed partitioning of the coordinates is chosen ahead of
time.
We show that this setting is too restrictive, constructing instances where
breaking locality by running
non-accelerated Gauss-Seidel with randomly sampled coordinates
substantially outperforms accelerated Gauss-Seidel with any
fixed partitioning.  Motivated by this finding, we analyze the accelerated
block Gauss-Seidel algorithm in the random coordinate sampling setting. Our
analysis captures the benefit of acceleration with a new data-dependent
parameter which is well behaved when the matrix sub-blocks are
well-conditioned.
Empirically, we show that accelerated Gauss-Seidel with random coordinate
sampling provides speedups for large scale machine learning tasks when
compared to non-accelerated Gauss-Seidel and the classical conjugate-gradient
algorithm.

\end{abstract}

%% file: intro.tex
\section{Introduction}
\label{sec:intro}

The randomized Gauss-Seidel method is a commonly used iterative algorithm to
compute the solution of an $n \times n$ linear system $Ax = b$ by updating a
single coordinate at a time in a randomized order. While this approach is known
to converge linearly to the true solution when $A$ is
positive definite (see e.g. \cite{leventhal10}), in practice it is often more
efficient to update a small block of coordinates at a time due to the effects
of cache locality.

In extending randomized Gauss-Seidel to the block setting, a natural question
that arises is how one should sample the next block.  At one extreme a
\emph{fixed partition} of the coordinates is chosen ahead of time. The
algorithm is restricted to randomly selecting blocks from this fixed partitioning,
thus favoring data locality.  At the other extreme we break locality by
sampling a new set of \emph{random coordinates} to form a block at every
iteration.

Theoretically, the fixed partition case is well understood both
for Gauss-Seidel~\cite{qu15,gower15} and its Nesterov accelerated variant~\cite{nesterov16}.
More specifically, at most $O(\mu_{\mathrm{part}}^{-1} \log(1/\varepsilon))$
iterations of Gauss-Seidel are sufficient to reach a solution with at most
$\varepsilon$ error, where $\mu_{\mathrm{part}}$ is a quantity which measures
how well the $A$ matrix is preconditioned by the block diagonal matrix containing the sub-blocks corresponding to the fixed
partitioning.
When acceleration is used, Nesterov and Stich~\cite{nesterov16} show that the
rate improves to
$O\left(\sqrt{\frac{n}{p} \mu_{\mathrm{part}}^{-1}}\log(1/\varepsilon)\right)$, where
$p$ is the partition size.

For the random coordinate selection model, the existing literature is less complete.
While it is known~\cite{qu15,gower15} that the iteration complexity with
random coordinate section is $O(\mu_{\mathrm{rand}}^{-1}\log(1/\varepsilon))$
for an $\varepsilon$ error solution,
$\mu_{\mathrm{rand}}$ is another instance dependent
quantity which is not directly comparable to $\mu_{\mathrm{part}}$.
Hence it is not obvious how much better, if at all, one expects
random coordinate selection to perform compared to fixed partitioning.

Our first contribution in this paper is to show that,
when compared to the random coordinate selection model, the fixed partition model
can perform very poorly in terms of iteration complexity to reach a
pre-specified error.
Specifically, we present a family of instances (similar to the matrices
recently studied by Lee and Wright~\cite{lee16})
where \emph{non}-accelerated Gauss-Seidel
with random coordinate selection performs \emph{arbitrarily}
faster than
both non-accelerated and even accelerated Gauss-Seidel, using \emph{any} fixed
partition. Our result thus shows the importance of the sampling strategy and that
acceleration cannot make up for a poor choice of sampling distribution.

This finding motivates us to further study the benefits
of acceleration under the random coordinate selection model.
Interestingly, the benefits are more nuanced under this model.
We show that acceleration improves the rate from
$O(\mu_{\mathrm{rand}}^{-1} \log(1/\varepsilon))$ to
$O\left(\sqrt{\nu \mu_{\mathrm{rand}}^{-1}} \log(1/\varepsilon)\right)$, where $\nu$ is a new instance
dependent quantity that satisfies $\nu \leq \mu_{\mathrm{rand}}^{-1}$.
We derive a bound on $\nu$
which suggests that if the sub-blocks of $A$ are all well conditioned, then
acceleration can provide substantial speedups.
We note that this is merely a sufficient condition, and our experiments suggest that
our bound is conservative.

In the process of deriving our results, we also develop
a general proof framework for randomized accelerated methods based on Wilson et al.~\cite{wilson16}
which avoids the use
of estimate sequences in favor of an explicit Lyapunov function. Using our
proof framework we are able to recover recent results~\cite{nesterov16,allenzhu16} on
accelerated coordinate descent. Furthermore, our proof framework allows us to immediately transfer our results
on Gauss-Seidel over to the randomized accelerated Kaczmarz algorithm, extending
a recent result by Liu and Wright~\cite{liu16} on updating a single constraint at a time to the block case.

Finally, we empirically demonstrate that despite its theoretical nuances,
accelerated Gauss-Seidel using random coordinate selection can provide significant
speedups in practical applications over Gauss-Seidel with fixed partition sampling,
as well as the classical conjugate-gradient (CG) algorithm.
As an example, for a kernel ridge regression (KRR) task in machine learning on
the augmented CIFAR-10 dataset ($n=250,000$), acceleration with random coordinate sampling performs up
to $1.5 \times$ faster than acceleration with a fixed partitioning to reach an
error tolerance of $10^{-2}$, with the gap substantially widening for smaller
error tolerances.  Furthermore, it performs over $3.5\times$ faster than
conjugate-gradient on the same task.

%% file: background.tex
\section{Background}
\label{sec:background}

%In this section, we develop the relevant context to explain our results
%in more detail.
%
We assume that we are given an $n \times n$ matrix $A$ which is positive
definite, and an $n$ dimensional response vector $b$. We also fix an integer
$p$ which denotes a block size.
%To avoid special edge cases we will
%assume that $1 < p < n$. We will also assume for simplicity that
%$n$ is an integer multiple of $p$.
%
Under the assumption of $A$ being positive definite, the function
$f(x) = \frac{1}{2} x^\T A x - x^\T b$
is strongly convex and smooth. Recent analysis of
Gauss-Seidel~\cite{gower15} proceeds by noting the connection between
Gauss-Seidel and (block) coordinate descent on $f$. This is the point of view
we will take in this paper.

\subsection{Existing rates for randomized block Gauss-Seidel}
\label{sec:background:existing_rates_gs}

We first describe the sketching framework of \cite{qu15,gower15} and show how
it yields rates on Gauss-Seidel when blocks are chosen via a fixed partition or
randomly at every iteration.  While we will only focus on the special case when
the sketch matrix represents column sampling, the sketching framework allows us
to provide a unified analysis of both cases.

%We introduce the sketching framework solely because it allows us to
%provide a unified analysis of both cases;
%many of our structural results will be stated
%in terms of the sketching framework.

To be more precise, let $\mathcal{D}$ be a distribution over $\R^{n \times p}$, and let $S_k \sim \mathcal{D}$ be drawn iid from $\mathcal{D}$.
If we perform block coordinate descent by minimizing $f$ along the range of $S_k$,
then the randomized block Gauss-Seidel update is given by
\begin{align}
    x_{k+1} = x_k - S_k(S_k^\T A S_k)^{\dag} S_k^\T (A x_k - b) \:. \label{eq:rand_gs}
\end{align}

\paragraph{Column sampling.}

Every index set $J \subseteq 2^{[n]}$ with $\abs{J}=p$ induces a sketching matrix
$S(J) = (e_{J(1)}, ..., e_{J(p)})$ where $e_i$ denotes the $i$-th standard basis vector
in $\R^{n}$, and $J(1), ..., J(p)$ is any ordering of the elements of $J$.
%%% Removing this footnote
%~\footnote{
%Note that the ordering is irrelevant because the quantity
%$S(J) (S(J)^\T A S(J))^{\dag} S(J)^\T$ is invariant under any ordering.}.
%
By equipping different probability measures on $2^{[n]}$, one can
easily describe fixed partition sampling as well as random coordinate sampling
(and many other sampling schemes).
The former puts uniform mass on the index sets $J_1, ..., J_{n/p}$,
whereas the latter puts uniform mass on all ${n \choose p}$ index sets of size $p$.
Furthermore, in the sketching framework there is no limitation to use a uniform distribution, nor is
there any limitation to use a fixed $p$ for every iteration.
For this paper, however, we will restrict our attention to these cases.

\paragraph{Existing rates.}

Under the assumptions stated above, \cite{qu15,gower15} show
that for every $k \geq 0$, the sequence \eqref{eq:rand_gs} satisfies
\begin{align}
    \E[\norm{ x_k - x_* }_{A}] &\leq (1 - \mu)^{k/2} \norm{ x_0 - x_* }_{A} \:,  \label{eq:rand_gs_rate}
\end{align}
where $\mu = \lambda_{\min}(\E[ P_{A^{1/2} S} ] )$.
The expectation in \eqref{eq:rand_gs_rate} is taken with respect to the randomness
of $S_0, S_1, ...$, and the expectation in the definition of $\mu$
is taken with respect to $S \sim \mathcal{D}$.
Under both fixed partitioning and random coordinate selection,
$\mu > 0$ is guaranteed (see e.g. \cite{gower15}, Lemma~4.3).
Thus, \eqref{eq:rand_gs} achieves a linear rate of convergence to the true solution,
with the rate governed by the $\mu$ quantity shown above.

We now specialize \eqref{eq:rand_gs_rate} to fixed partitioning and random
coordinate sampling, and provide some intuition for why we expect the latter
to outperform the former in terms of iteration complexity.
We first consider the case when the sampling distribution
corresponds to fixed partitioning. Assume for notational convenience that the
fixed partitioning corresponds to
placing the first $p$ coordinates in the first partition $J_1$,
the next $p$ coordinates in the second partition $J_2$, and so on.
Here, $\mu = \mu_{\mathrm{part}}$ corresponds to a measure of how
close the product of $A$ with
the inverse of the block diagonal is to the identity matrix, defined as
\begin{align}
    %\mu_{\mathrm{part}} = \frac{p}{n} \lambda_{\min}\left(A \begin{bmatrix} A_{J_1}^{-1} & 0 & ... & 0 \\
    %0 & A_{J_2}^{-1} & ... & 0 \\ \vdots & \vdots & \ddots & \vdots \\ 0 & 0 & ... & A_{J_{n/p}}^{-1} \end{bmatrix} \right) \:. \label{eq:mu_part}
    \mu_{\mathrm{part}} = \frac{p}{n} \lambda_{\min}\left(A \cdot \mathrm{blkdiag}\left( A_{J_1}^{-1}, ..., A_{J_{n/p}}^{-1} \right) \right) \:. \label{eq:mu_part}
\end{align}
Above, $A_{J_i}$ denotes the $p \times p$ matrix corresponding to the sub-matrix
of $A$ indexed by the $i$-th partition.
A loose lower bound on $\mu_{\mathrm{part}}$ is
\begin{align}
    \mu_{\mathrm{part}} \geq \frac{p}{n} \frac{\lambda_{\min}(A)}{\max_{1 \leq i \leq n/p} \lambda_{\max}(A_{J_i})} \:. \label{eq:mu_part_crude_lower_bound}
\end{align}
On the other hand, in the random coordinate case,
Qu et al.~\cite{qu15} derive a lower bound on $\mu = \mu_{\mathrm{rand}}$ as
\begin{align}
    \mu_{\mathrm{rand}} \geq \frac{p}{n} \left( \beta + (1-\beta) \frac{\max_{1 \leq i \leq n} A_{ii}}{\lambda_{\min}(A)} \right)^{-1} \:, \label{eq:mu_rand_lower_bound}
\end{align}
where \mbox{$\beta = (p-1)/(n-1)$}.
Using the lower bounds
\eqref{eq:mu_part_crude_lower_bound} and
\eqref{eq:mu_rand_lower_bound},
we can upper bound the iteration complexity of
fixed partition Gauss-Seidel $N_{\mathrm{part}}$
by $O\left( \frac{n}{p} \frac{\max_{1 \leq i \leq n/p} \lambda_{\max}(A_{J_i})}{\lambda_{\min}(A)} \log(1/\varepsilon) \right)$
and random coordinate Gauss-Seidel $N_{\mathrm{rand}}$ as
$O\left( \frac{n}{p} \frac{\max_{1 \leq i \leq n} A_{ii}}{\lambda_{\min}(A)}  \log(1/\varepsilon) \right)$.
Comparing the bound on $N_{\mathrm{part}}$ to the bound on $N_{\mathrm{rand}}$,
it is not unreasonable to expect that random
coordinate sampling has better iteration complexity than fixed partition
sampling in certain cases. In Section~\ref{sec:results}, we verify this by
constructing instances $A$ such that
fixed partition Gauss-Seidel takes arbitrarily more
iterations to reach a pre-specified error tolerance
compared with random coordinate Gauss-Seidel.

\subsection{Accelerated rates for fixed partition Gauss-Seidel}

%We now discuss previous work on accelerating Gauss-Seidel.
%
Based on the interpretation of Gauss-Seidel as block coordinate descent on the
function $f$, we can use Theorem~1 of Nesterov and Stich~\cite{nesterov16} to
recover a procedure and a rate for accelerating \eqref{eq:rand_gs} in the fixed
partition case; the specific details are discussed in
Section~\ref{sec:proofs:relating_acdm} of
the appendix. We will refer to this procedure as ACDM.

The convergence guarantee of
the ACDM procedure
is that for all $k \geq 0$,
\begin{align}
    \E[\norm{x_k - x_*}_{A}] \leq O\left( \left(1-\sqrt{\frac{p}{n} \mu_{\mathrm{part}}}\right)^{k/2} \norm{x_0 - x_*}_{A} \right) \:. \label{eq:rand_acc_nesterov}
\end{align}
Above, $\mu_{\mathrm{part}}$ is the same quantity defined in
\eqref{eq:mu_part}.
Comparing \eqref{eq:rand_acc_nesterov} to the non-accelerated Gauss-Seidel rate given in
\eqref{eq:rand_gs_rate},
we see that acceleration improves the iteration complexity to reach a solution with $\varepsilon$ error
from
$O(\mu_{\mathrm{part}}^{-1} \log(1/\varepsilon))$ to
$O\left(\sqrt{\frac{n}{p} \mu_{\mathrm{part}}^{-1}} \log(1/\varepsilon)\right)$, as discussed in
Section~\ref{sec:intro}.

%% file: results.tex
\section{Results}
\label{sec:results}

We now present the main results of the paper.
All proofs are deferred to the appendix.

\subsection{Fixed partition vs random coordinate sampling}
\label{sec:results:separation}

Our first result is to construct instances where
Gauss-Seidel with fixed partition sampling
runs arbitrarily slower than random coordinate sampling, even if acceleration
is used.

Consider the family of $n \times n$ positive definite matrices $\mathscr{A}$
given by
$\mathscr{A} = \{ A_{\alpha,\beta} : \alpha > 0, \alpha + \beta > 0 \}$ with
$A_{\alpha,\beta}$ defined as
$A_{\alpha,\beta} = \alpha I + \frac{\beta}{n} \ind_n\ind_n^\T$.
The family $\mathscr{A}$ exhibits a crucial property that $\Pi^\T A_{\alpha,\beta}
\Pi = A_{\alpha,\beta}$ for every $n \times n$ permutation matrix $\Pi$.
Lee and Wright~\cite{lee16} recently exploited this invariance to
illustrate the behavior of cyclic versus randomized permutations in coordinate
descent.

We explore the behavior of Gauss-Seidel as the matrices $A_{\alpha,\beta}$ become
ill-conditioned.  To do this, we consider a particular parameterization
which holds the minimum eigenvalue equal to one and sends the maximum eigenvalue
to infinity via the sub-family
$\{ A_{1,\beta} \}_{\beta > 0}$.
Our first proposition characterizes the behavior of Gauss-Seidel with fixed partitions on
this sub-family.
\begin{proposition}
\label{prop:simple_calcs_model2_fixed}
Fix $\beta > 0$ and positive integers $n,p,k$ such that $n = pk$.
Let $\{J_i\}_{i=1}^{k}$ be any partition of $\{1, ..., n\}$ with $\abs{J_i} = p$,
and denote $S_i \in \R^{n \times p}$ as the column selector for partition $J_i$.
Suppose $S \in \R^{n \times p}$ takes on value $S_i$ with probability $1/k$.
For every $A_{1,\beta} \in \mathscr{A}$ we have that
\begin{align}
    \mu_{\mathrm{part}} &= \frac{p}{n+\beta p} \label{eq:lambda_min_model2_fixed} \:.
\end{align}
\end{proposition}
Next, we perform a similar calculation under the
random column sampling model.
\begin{proposition}
\label{prop:simple_calcs_model2_random}
Fix $\beta > 0$ and integers $n,p$ such that $1 < p < n$.  Suppose
each column of $S \in \R^{n \times p}$
is chosen uniformly at random from $\{e_1, ..., e_n\}$ without
replacement.  For every $A_{1,\beta} \in \mathscr{A}$ we have that
\begin{align}
    \mu_{\mathrm{rand}} &= \frac{p}{n+\beta p} + \frac{(p-1)\beta p}{(n-1)(n+\beta p)}  \label{eq:lambda_min_model2_random} \:.
\end{align}
\end{proposition}
The differences between \eqref{eq:lambda_min_model2_fixed} and
\eqref{eq:lambda_min_model2_random} are striking.  Let us assume that $\beta$ is
much larger than $n$.
Then, we have that $ \mu_{\mathrm{part}} \approx 1/\beta$ for
\eqref{eq:lambda_min_model2_fixed}, whereas $ \mu_{\mathrm{rand}} \approx
\frac{p-1}{n-1}$ for \eqref{eq:lambda_min_model2_random}.  That is,
$\mu_{\mathrm{part}}$ can be made arbitrarily smaller than $\mu_{\mathrm{rand}}$
as $\beta$ grows.

Our next proposition states that
the rate of Gauss-Seidel from \eqref{eq:rand_gs_rate} is tight order-wise
in that for any instance there always exists a starting point
which saturates the bound.
\begin{proposition}
\label{prop:lower_bound_gs}
Let $A$ be an $n \times n$ positive definite matrix, and let
$S$ be a random matrix such that $\mu = \lambda_{\min}(\E[P_{A^{1/2} S}]) > 0$.
Let $x_*$ denote the solution to $A x = b$.
There exists a starting point $x_0 \in \R^{n}$ such that
the sequence \eqref{eq:rand_gs} satisfies for all $k \geq 0$,
\begin{align}
    \E[\norm{x_k - x_*}_{A}] \geq (1-\mu)^{k} \norm{x_0 - x_*}_{A} \:. \label{eq:rand_gs_lb}
\end{align}
\end{proposition}
From \eqref{eq:rand_gs_rate} we see that
Gauss-Seidel using random coordinates computes a
solution $x_k$ satisfying $\E[\norm{ x_k - x_* }_{A_{1,\beta}}] \leq \varepsilon$
in at most $k = O(\frac{n}{p}\log(1/\varepsilon))$ iterations.
On the other hand, Proposition~\ref{prop:lower_bound_gs}
states that for any fixed partition, there exists an input $x_0$ such that
$k = \Omega(\beta \log(1/\varepsilon))$ iterations are required to
reach the same $\varepsilon$ error tolerance.
Furthermore, the situation does not improve even if
use ACDM from \cite{nesterov16}.
Proposition~\ref{prop:lower_bound_acc_gs}, which we
describe later, implies that for any fixed partition there exists
an input $x_0$ such that
$k = \Omega\left(\sqrt{\frac{n}{p} \beta} \log(1/\varepsilon)\right)$ iterations
are required for ACDM
to reach $\varepsilon$ error. Hence as $\beta \longrightarrow \infty$, the gap
between random coordinate and fixed partitioning can be made arbitrarily
large.
These findings are numerically verified in Section~\ref{subsec:fixed_vs_random}.

\subsection{A Lyapunov analysis of accelerated Gauss-Seidel and Kaczmarz}
\label{sec:results:convergence}

Motivated by our findings,
our goal is to understand the behavior of accelerated Gauss-Seidel under random
coordinate sampling. In order to do this, we establish a general framework from which
the behavior of accelerated Gauss-Seidel with random coordinate sampling
follows immediately, along with rates for accelerated randomized Kaczmarz
\cite{liu16} and the accelerated coordinate descent methods of
\cite{nesterov16} and \cite{allenzhu16}.

For conciseness, we describe a simpler version of our framework which
is still able to capture both the Gauss-Seidel and Kaczmarz results,
deferring the general version to the full version of the paper.
Our general result requires a bit more notation, but follows the same line of
reasoning.

Let $H$ be a random $n \times n$ positive semi-definite matrix.  Put $G =
\E[H]$, and suppose that $G$ exists and is positive definite.  Furthermore, let
$f : \R^{n} \longrightarrow \R$ be strongly convex and
smooth, and let $\mu$ denote the strong convexity constant of $f$ w.r.t. the
$\norm{\cdot}_{G^{-1}}$ norm.

Consider the following sequence $\{(x_k, y_k, z_k)\}_{k\geq 0}$
defined by the recurrence
\begin{subequations}
    \label{eq:simple:acc}
    \begin{align}
        x_{k+1} &= \frac{1}{1+\tau} y_k + \frac{\tau}{1+\tau} z_k \:, \label{eq:simple:acc:coupling} \\
        y_{k+1} &= x_{k+1} - H_k \nabla f(x_{k+1}) \:, \label{eq:simple:acc:gradient_step} \\
        z_{k+1} &= z_k + \tau ( x_{k+1} - z_k ) - \frac{\tau}{\mu} H_k \nabla f(x_{k+1}) \:, \label{eq:simple:acc:zseq}
    \end{align}
\end{subequations}
where $H_0, H_1, ...$ are independent realizations of $H$
and $\tau$ is a parameter to be chosen.
Following \cite{wilson16}, we construct a candidate Lyapunov function $V_k$
for the sequence \eqref{eq:simple:acc} defined as
\begin{align}
  V_k = f(y_k) - f_* + \frac{\mu}{2} \norm{z_k - x_*}^2_{G^{-1}} \:.
\end{align}
The following theorem demonstrates that $V_k$ is indeed
a Lyapunov function for $(x_k, y_k, z_k)$.
\begin{theorem}
\label{thm:general_simplified}
Let $f, G, H$ be as defined above.
Suppose further that $f$ has 1-Lipschitz gradients w.r.t. the $\norm{\cdot}_{G^{-1}}$ norm,
and for every fixed $x \in \R^n$,
  \begin{align}
    f(\Phi(x;H)) \leq f(x) - \frac{1}{2} \norm{\nabla f(x)}^2_{H} \:, \label{eq:grad_progress_assumption}
  \end{align}
holds for a.e. $H$, where $\Phi(x; H) = x - H \nabla f(x)$.
Set $\tau$ in \eqref{eq:simple:acc} as $\tau = \sqrt{\mu/\nu}$, with
  \begin{align*}
    \nu = \lambda_{\max}\left( \E\left[ (G^{-1/2} H G^{-1/2} )^2 \right] \right) \:.
  \end{align*}
Then for every $k \geq 0$, we have
  \begin{align*}
    \E[ V_k ] \leq (1-\tau)^{k} V_0 \:.
  \end{align*}
\end{theorem}

We now proceed to specialize Theorem~\ref{thm:general_simplified} to both
the Gauss-Seidel and Kaczmarz settings.

\subsubsection{Accelerated Gauss-Seidel}
Let $S \in \R^{n \times p}$ denote a random sketching matrix.
As suggested in Section~\ref{sec:background}, we set
$f(x) = \frac{1}{2} x^\T A x - x^\T b$ and put $H = S(S^\T A S)^{\dag} S^\T$.
Note that $G = \E[ S(S^\T A S)^{\dag} S^\T ]$ is positive definite
iff $\lambda_{\min}(\E[ P_{A^{1/2} S}]) > 0$,
and is hence satisfied for both fixed partition and random coordinate sampling
(c.f. Section~\ref{sec:background}).
Next, the fact that $f$ is 1-Lipschitz w.r.t. the $\norm{\cdot}_{G^{-1}}$ norm
and the condition \eqref{eq:grad_progress_assumption} are standard calculations.
All the hypotheses of Theorem~\ref{thm:general_simplified} are thus satisfied, and
the conclusion
is Theorem~\ref{thm:rand_acc_gs}, which characterizes the
rate of convergence for accelerated Gauss-Seidel
(Algorithm~\ref{alg:rand_acc_gs}).

\begin{algorithm}[h!]
\caption{Accelerated randomized block Gauss-Seidel.}
\label{alg:rand_acc_gs}
\begin{algorithmic}[1]
  \REQUIRE{$A \in \R^{n \times n}$, $A \succ 0$, $b \in \R^{n}$, sketching matrices $\{S_k\}_{k=0}^{T-1} \subseteq \R^{n \times p}$, $x_0 \in \R^{n}$, $\mu \in (0,1)$, $\nu \geq 1$.}
  \STATE Set $\tau = \sqrt{\mu/\nu}$.
  \STATE Set $y_0 = z_0 = x_0$.
\FOR{$k=0, ..., T-1$}
    \STATE $x_{k+1} = \frac{1}{1+\tau} y_k + \frac{\tau}{1+\tau} z_k$.
    \STATE $H_k = S_k(S_k^\T A S_k)^{\dag} S_k^\T$.
    \STATE $y_{k+1} = x_{k+1} -  H_k (A x_{k+1} - b)$. \label{alg:line:localupdate}
    \STATE $z_{k+1} = z_k + \tau(x_{k+1} - z_k) - \frac{\tau}{\mu} H_k (A x_{k+1} - b)$.
\ENDFOR
\STATE Return $y_T$.
\end{algorithmic}
\end{algorithm}

\begin{theorem}
\label{thm:rand_acc_gs}
Let $A$ be an $n \times n$ positive definite matrix and $b \in \R^{n}$.
Let $x_* \in \R^{n}$ denote the unique vector satisfying $A x_* = b$.
Suppose each $S_k$, $k =0, 1, 2, ...$ is an independent copy of a random matrix $S \in \R^{n \times p}$.
Put $\mu = \lambda_{\min}( \E[P_{A^{1/2} S}] )$, and suppose the distribution of $S$
satisfies $\mu > 0$.
Invoke Algorithm~\ref{alg:rand_acc_gs} with $\mu$ and $\nu$, where
\begin{align}
    \nu = \lambda_{\max}\left( \E\left[ (G^{-1/2} H G^{-1/2})^2 \right] \right) \:, \label{eq:gs_variance_term}
\end{align}
with $H = S(S^\T A S)^{\dag}S^\T$ and $G = \E[H]$.
Then with $\tau = \sqrt{\mu/\nu}$, for all $k \geq 0$,
\begin{align}
    \E[\norm{ y_k - x_* }_A] \leq \sqrt{2} (1 - \tau)^{k/2} \norm{ x_0 - x_* }_A \:. \label{eq:rand_acc_gs_rate}
\end{align}
\end{theorem}

Note that in the setting of Theorem~\ref{thm:rand_acc_gs}, by
the definition of $\nu$ and $\mu$,
it is always the case that $\nu \leq 1/\mu$. Therefore, the
iteration complexity of acceleration is at least as good as the iteration
complexity without acceleration.

We conclude our discussion of Gauss-Seidel
by describing the analogue of Proposition~\ref{prop:lower_bound_gs}
for Algorithm~\ref{alg:rand_acc_gs}, which shows that
our analysis in Theorem~\ref{thm:rand_acc_gs} is tight order-wise.
The following proposition applies to
ACDM as well; we show in the full version of the paper how ACDM
can be viewed as a special case of Algorithm~\ref{alg:rand_acc_gs}.
\begin{proposition}
\label{prop:lower_bound_acc_gs}
Under the setting of Theorem~\ref{thm:rand_acc_gs}, there exists
starting positions $y_0, z_0 \in \R^{n}$ such that
the iterates $\{(y_k, z_k)\}_{k \geq 0}$ produced by Algorithm~\ref{alg:rand_acc_gs} satisfy
    \begin{align*}
        \E[\norm{y_k - x_*}_{A} + \norm{z_k - x_*}_{A}] \geq (1-\tau)^k \norm{y_0 - x_*}_{A} \:.
    \end{align*}
\end{proposition}

\subsubsection{Accelerated Kaczmarz}

The argument for Theorem~\ref{thm:rand_acc_gs} can be slightly modified
to yield a result for randomized accelerated Kaczmarz in the sketching
framework, for the case of a consistent overdetermined linear system.

Specifically, suppose we are given an $m \times n$ matrix $A$ which has
full column rank, and $b \in \mathcal{R}(A)$.
Our goal is to recover the unique $x_*$ satisfying $A x_* = b$.
To do this, we apply a similar line of reasoning as
\cite{lee13}.
We set $f(x) = \frac{1}{2}\norm{x - x_*}^2_2$ and
$H = P_{A^\T S}$, where $S$ again is our random sketching matrix.
At first, it appears our choice of $f$ is problematic since we do not have access to
$f$ and $\nabla f$, but a quick calculation shows that
$H \nabla f(x) = (S^\T A)^{\dag} S^\T (A x - b)$.
Hence, with $r_k = A x_k - b$, the sequence \eqref{eq:simple:acc} simplifies to
\begin{subequations}
  \label{eq:kaczmarz}
    \begin{align}
        x_{k+1} &= \frac{1}{1+\tau} y_k + \frac{\tau}{1+\tau} z_k \:, \\
        y_{k+1} &= x_{k+1} - (S_k^\T A)^{\dag} S_k^\T r_{k+1} \:, \\
        z_{k+1} &= z_k + \tau ( x_{k+1} - z_k ) - \frac{\tau}{\mu} (S_k^\T A)^{\dag} S_k^\T r_{k+1} \:.
    \end{align}
\end{subequations}
The remainder of the argument
proceeds nearly identically, and leads to the following theorem.
\begin{theorem}
\label{thm:rand_acc_rk_simple}
Let $A$ be an $m \times n$ matrix with full column rank, and $b = A x_*$.
Suppose each $S_k$, $k =0, 1, 2, ...$ is an independent copy of a random sketching matrix $S \in \R^{m \times p}$.
Put $H = P_{A^\T S}$ and $G = \E[H]$.
The sequence \eqref{eq:kaczmarz}
with $\mu = \lambda_{\min}( \E[P_{A^\T S}] )$,
$\nu = \lambda_{\max}\left( \E\left[ (G^{-1/2} H G^{-1/2})^2 \right] \right)$,
and $\tau = \sqrt{\mu/\nu}$ satisfies for all $k \geq 0$,
\begin{align}
    \E[\twonorm{ y_k - x_* }] \leq \sqrt{2} (1 - \tau)^{k/2} \twonorm{ x_0 - x_* } \:. \label{eq:rand_acc_rk_rate}
\end{align}
\end{theorem}

Specialized to the setting of \cite{liu16} where each row of $A$
has unit norm and is sampled uniformly at every iteration,
it can be shown
(Section~\ref{sec:proofs:compute_nu_mu_liu_wright})
that $\nu \leq m$ and $\mu = \frac{1}{m} \lambda_{\min}(A^\T A)$.
Hence, the above theorem states that
the iteration complexity to reach $\varepsilon$ error
is $O\left(\frac{m}{\sqrt{\lambda_{\min}(A^\T A)}} \log(1/\varepsilon)\right)$,
which matches Theorem~5.1 of \cite{liu16} order-wise.
%On the other hand,
However, Theorem~\ref{thm:rand_acc_rk_simple} applies in general
for any sketching matrix.

\subsection{Specializing accelerated Gauss-Seidel to random coordinate sampling}
\label{sec:results:rand_coords}

We now instantiate Theorem~\ref{thm:rand_acc_gs} to
random coordinate sampling.  The $\mu$ quantity which appears
in Theorem~\ref{thm:rand_acc_gs} is identical to the quantity appearing
in the rate \eqref{eq:rand_gs_rate} of non-accelerated Gauss-Seidel.
That is, the iteration complexity to reach tolerance $\varepsilon$ is
$O\left(\sqrt{\nu \mu_{\mathrm{rand}}^{-1}} \log(1/\varepsilon)\right)$, and the only new term here is $\nu$.
In order to provide a more intuitive interpretation of the $\nu$ quantity,
we present an upper bound on $\nu$ in terms of an effective block condition number
defined as follows.
Given an index set $J \subseteq 2^{[n]}$, define the effective block condition number of
a matrix $A$ as
$\bigkappa_{\mathrm{eff},J}(A) = \frac{\max_{i \in J} A_{ii}}{\lambda_{\min}(A_{J})}$.
Note that $\bigkappa_{\mathrm{eff},J}(A) \leq \bigkappa(A_{J})$ always.
The following lemma gives upper and lower bounds on the $\nu$ quantity.
\begin{lemma}
\label{lemma:nu_upper_bound}
Let $A$ be an $n \times n$ positive definite matrix and let $p$ satisfy $1 < p < n$.
We have that
    \begin{align*}
        \frac{n}{p} \leq \nu \leq \frac{n}{p} \left( \frac{p-1}{n-1} + \left(1-\frac{p-1}{n-1}\right) \bigkappa_{\mathrm{eff},p}(A) \right) \:,
    \end{align*}
where $\bigkappa_{\mathrm{eff},p}(A) = \max_{J \subseteq 2^{[n]} : \abs{J}=p} \bigkappa_{\mathrm{eff},J}(A)$,
$\nu$ is defined in \eqref{eq:gs_variance_term}, and the distribution of $S$
corresponds to uniformly selecting $p$ coordinates without replacement.
\end{lemma}

Lemma~\ref{lemma:nu_upper_bound} states that if the $p \times p$ sub-blocks of
$A$ are well-conditioned as defined by the effective block condition number
$\bigkappa_{\mathrm{eff},J}(A)$, then the speed-up of accelerated Gauss-Seidel with
random coordinate selection over its non-accelerate counterpart parallels the
case of fixed partitioning sampling (i.e.  the rate described in
\eqref{eq:rand_acc_nesterov} versus the rate in \eqref{eq:rand_gs_rate}).
This is a reasonable condition, since very ill-conditioned sub-blocks will lead to
numerical instabilities in solving the sub-problems when implementing Gauss-Seidel.
On the other hand, we note that Lemma~\ref{lemma:nu_upper_bound} provides merely a sufficient
condition for speed-ups from acceleration, and is conservative.
Our numerically experiments in Section~\ref{sec:experiments:mu_and_nu} suggest
that in many cases the $\nu$ parameter behaves closer to the lower bound $n/p$
than Lemma~\ref{lemma:nu_upper_bound} suggests.
We leave a more thorough theoretical analysis of this parameter
to future work.

We can now combine Theorem~\ref{thm:rand_acc_gs} with \eqref{eq:mu_rand_lower_bound}
to derive the following upper bound on the iteration complexity of
accelerated Gauss-Seidel with random coordinates as
\begin{align*}
  N_{\mathrm{rand,acc}} \leq O\left( \frac{n}{p} \sqrt{ \frac{\max_{1 \leq i \leq n} A_{ii}}{\lambda_{\min}(A)} \bigkappa_{\mathrm{eff},p}(A)} \log(1/\varepsilon) \right) \:.
\end{align*}

\paragraph{Illustrative example.}

We conclude our results by illustrating our bounds on a simple
example.
Consider the sub-family $\{A_{\delta}\}_{\delta > 0} \subseteq \mathscr{A}$, with
\begin{align}
    A_{\delta} = A_{n+\delta,-n} \:, \;\; \delta > 0 \:. \label{eq:model1}
\end{align}
A simple calculation yields that
$\bigkappa_{\mathrm{eff},p}(A_\delta) = \frac{n-1+\delta}{n-p+\delta}$,
and hence
Lemma~\ref{lemma:nu_upper_bound} states that
$\nu(A_{\delta}) \leq \frac{n}{p}\left(1 + \frac{p-1}{n-1}\right)$.
Furthermore, by a similar calculation to Proposition~\ref{prop:simple_calcs_model2_random},
$\mu_{\mathrm{rand}} = \frac{p\delta}{n(n-p+\delta)}$.
Assuming for simplicity that $p = o(n)$ and $\delta \in (0,1)$,
Theorem~\ref{thm:rand_acc_gs} states that
at most $O(\frac{n^{3/2}}{p\sqrt{\delta}}\log(1/\varepsilon))$ iterations are sufficient
for an $\varepsilon$-accurate solution.
On the other hand, without acceleration
\eqref{eq:rand_gs_rate} states that $O(\frac{n^2}{p\delta}\log(1/\varepsilon))$
iterations are sufficient and Proposition~\ref{prop:lower_bound_gs}
shows there exists a starting position for which it is necessary.
Hence, as either $n$ grows large or $\delta$ tends to zero,
the benefits of acceleration become more pronounced.

%% file: related.tex
\section{Related Work}
\label{sec:related}

We split the related work into two broad categories of interest: (a) work related 
to coordinate descent (CD) methods on convex functions
and (b) randomized solvers designed for solving consistent linear systems.

When $A$ is positive definite,
Gauss-Seidel can be interpreted as an instance of coordinate
descent on a strongly convex quadratic function. We therefore review related work
on both non-accelerated and accelerated coordinate descent, focusing
on the randomized setting instead of the more classical cyclic order or Gauss-Southwell rule
for selecting the next coordinate. See \cite{tseng09} for a discussion on
non-random selection rules, \cite{nutini15} for a comparison of
random selection versus Gauss-Southwell, and \cite{nutini16}
for efficient implementations of Gauss-Southwell.

Nesterov's original paper in~\cite{nesterov12} first considered randomized CD on
convex functions, assuming a partitioning of coordinates fixed ahead of time.
The analysis included both non-accelerated and accelerated variants for convex
functions.  This work sparked a resurgence of interest in CD methods for large
problems. Most relevant to our paper are extensions to the block
setting \cite{richtarik14}, handling arbitrary sampling distributions
\cite{qu14a, qu14b, fountoulakis16}, and second order updates for quadratic
functions \cite{qu16}. 

For accelerated CD, Lee and Sidford~\cite{lee13} generalize the analysis of Nesterov~\cite{nesterov12}.
While the analysis of \cite{lee13} was limited to
selecting a single coordinate at a time, several follow on works \cite{qu14a,
lin14, lu15, fercoq15} generalize to block and non-smooth settings.  More
recently, both Allen-Zhu et al.~\cite{allenzhu16} and Nesterov and Stich~\cite{nesterov16} independently improve
the results of \cite{lee13} by using a different non-uniform sampling
distribution. One of the most notable aspects of the analysis in
\cite{allenzhu16} is a departure from the (probabilistic) estimate sequence
framework of Nesterov. 
Instead, the authors construct a valid
Lyapunov function for coordinate descent, although they do not explicitly
mention this. 
In our work, we make this Lyapunov point of view explicit.
The constants in our acceleration updates arise
from a particular discretization and Lyapunov function outlined from
Wilson et al.~\cite{wilson16}. Using this framework makes our proof particularly transparent,
and allows us to recover results for strongly convex functions from \cite{allenzhu16}
and \cite{nesterov16} as a special case.

From the numerical analysis side 
both the Gauss-Seidel and Kaczmarz algorithm
are classical methods. Strohmer and Vershynin~\cite{strohmer09} were the first
to prove a linear rate of convergence for randomized Kaczmarz, and
Leventhal and Lewis~\cite{leventhal10} provide a similar kind of analysis 
for randomized Gauss-Seidel. Both of these were in the single 
constraint/coordinate setting. The block setting was later analyzed by Needell and Tropp~\cite{needell14}.
More recently, Gower and Richt{\'{a}}rik~\cite{gower15} provide a unified analysis
for both randomized block Gauss-Seidel and Kaczmarz in the sketching framework. We adopt this
framework in this paper.
Finally, Liu and Wright~\cite{liu16} provide an accelerated analysis of randomized Kaczmarz once again
in the single constraint setting and we extend this to the block setting. 

%% file: experiments.tex
\section{Experiments}
\label{sec:experiments}

In this section we experimentally validate our theoretical results on how our
accelerated algorithms can improve convergence rates. Our experiments use a
combination of synthetic matrices and matrices from large scale machine
learning tasks.

\textbf{Setup.} We run all our experiments on a 4 socket Intel Xeon CPU E7-8870 machine with 18
cores per socket and 1TB of DRAM. We implement all our algorithms in Python using \texttt{numpy},
and use the Intel MKL library with 72 OpenMP threads for numerical operations.
We report errors as relative errors,
i.e. $\norm{x_k - x_*}_{A}^2/\norm{x_*}_{A}^2$.
Finally, we use the best values of $\mu$ and $\nu$ found by tuning each experiment.

We implement fixed partitioning by creating random blocks of coordinates at the beginning of the
experiment and cache the corresponding matrix blocks to improve performance.
For random coordinate sampling, we select a new block of coordinates at each iteration.

For our fixed partition experiments, we restrict our attention to uniform sampling.
While Gower and Richt{\'{a}}rik~\cite{gower15} propose a non-uniform scheme based on $\Tr(S^\T A S)$,
for translation-invariant kernels this reduces to uniform sampling.
Furthermore, as the kernel block Lipschitz constants were also roughly the same, other
non-uniform schemes~\cite{allenzhu16} also reduce to nearly uniform sampling.

\subsection{Fixed partitioning vs random coordinate sampling}
\label{subsec:fixed_vs_random}

Our first set of experiments numerically verify the separation between
fixed partitioning sampling versus random coordinate sampling.

\begin{figure}[t!]
  \centering
  \begin{minipage}[t]{0.45\textwidth}
  \begin{center}
  \includegraphics[width=\columnwidth]{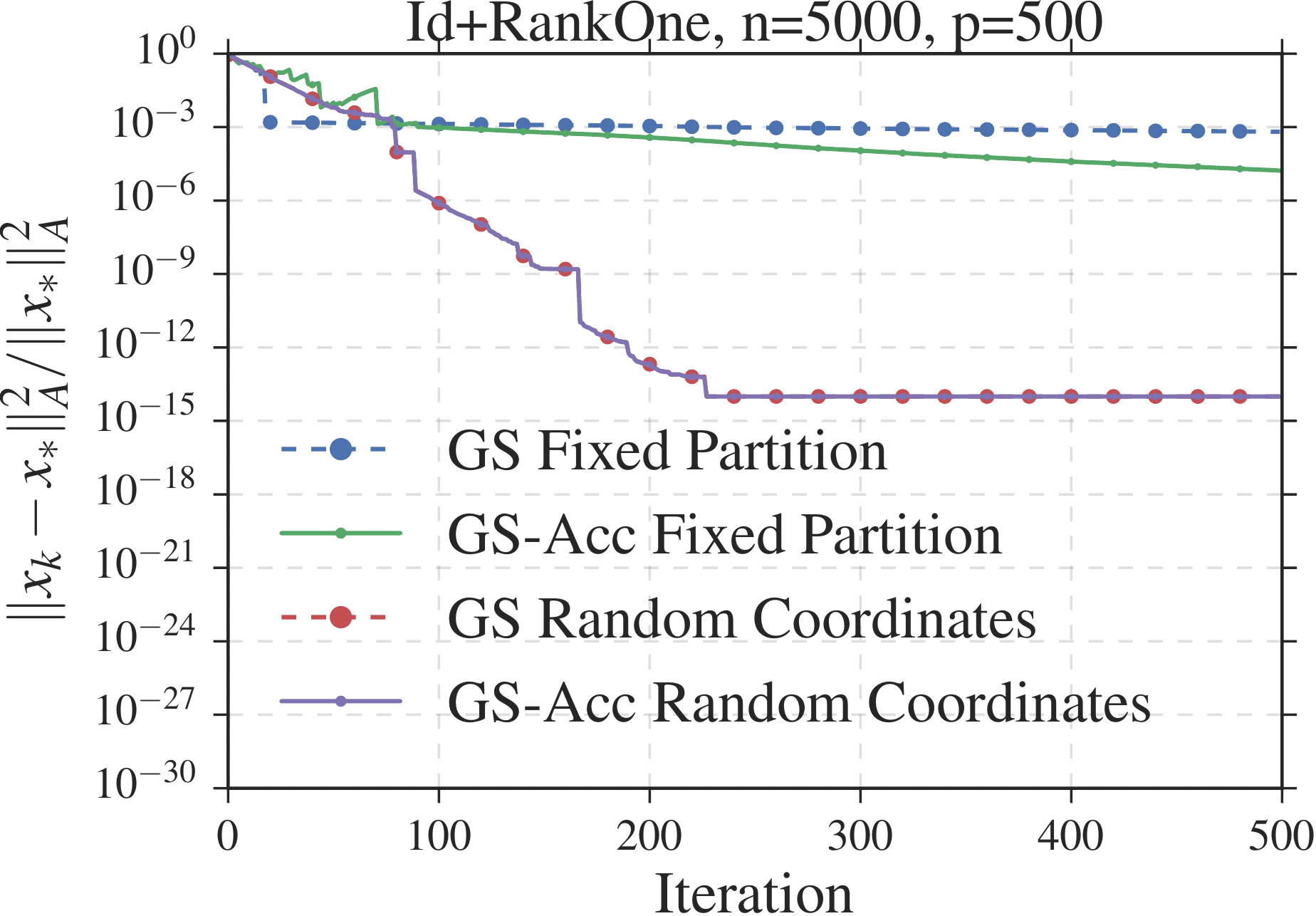}
  \end{center}
  \caption{Experiments comparing fixed partitions versus random coordinate sampling for the example from
Section~\ref{sec:results:separation} with $n=5000$ coordinates, block size $p=500$.}
    \label{fig:toy_parts_vs_coords}
  \end{minipage}
  \hspace{.06\textwidth}
    \begin{minipage}[t]{0.45\textwidth}
  \begin{center}
    \includegraphics[width=\columnwidth]{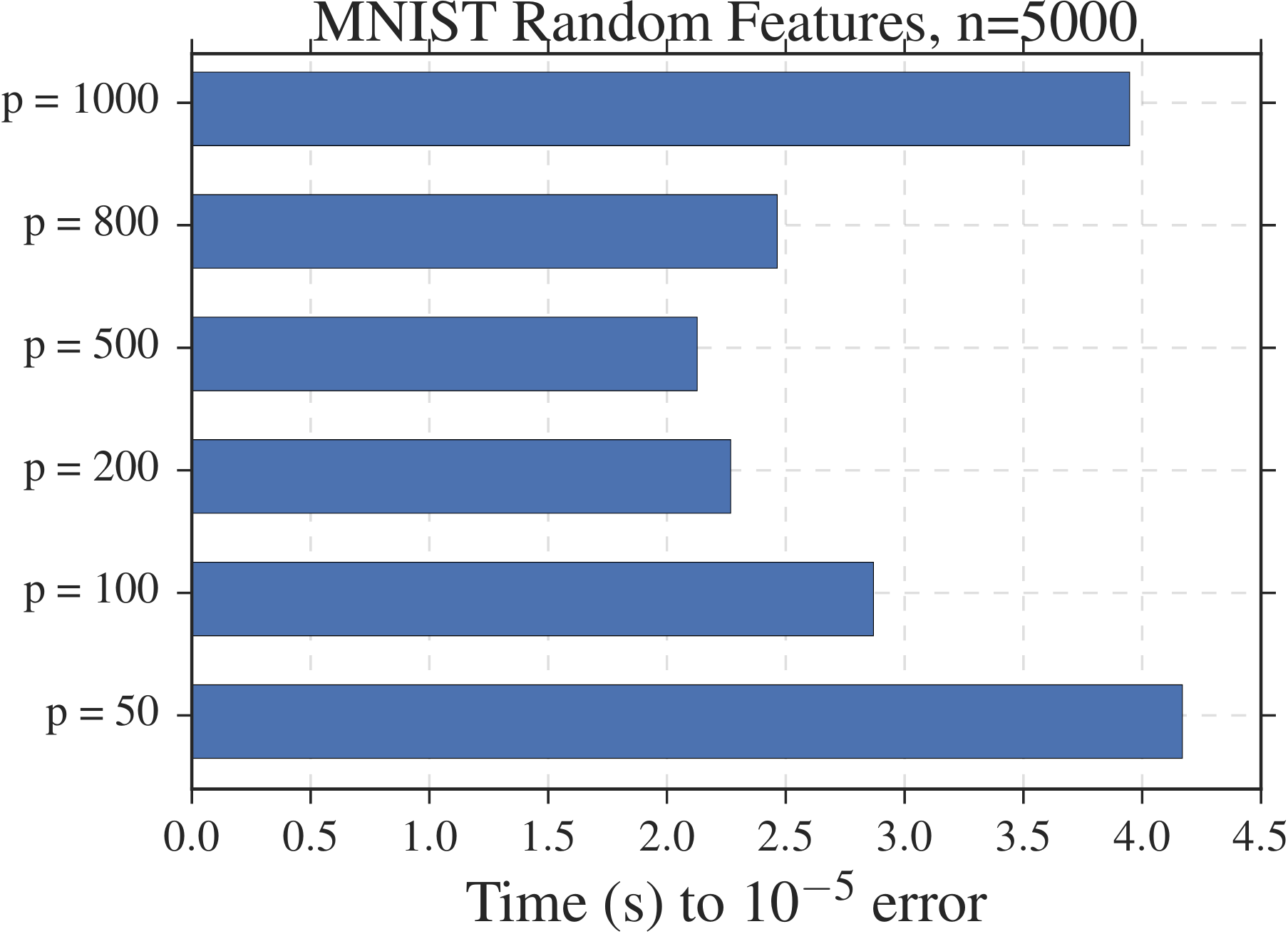}
  \end{center}
  \caption{The effect of block size on the accelerated Gauss-Seidel method. For the MNIST dataset
(pre-processed using random features) we see that block size of $p=500$ works best.}
    \label{fig:block_size_effect}
    \end{minipage}
\end{figure}

Figure~\ref{fig:toy_parts_vs_coords} shows the progress per iteration on
solving $A_{1,\beta} x = b$, with the $A_{1,\beta}$ defined
in Section~\ref{sec:results:separation}. Here
we set $n=5000$, $p=500$, $\beta=1000$, and $b \sim N(0, I)$.
Figure~\ref{fig:toy_parts_vs_coords} verifies our analytical findings in
Section~\ref{sec:results:separation}, that the fixed partition scheme
is substantially worse than uniform sampling on this instance.
It also shows that in this case, acceleration provides little benefit in the
case of random coordinate sampling.
This is because both $\mu$ and $1/\nu$ are order-wise $p/n$, and
hence the rate for accelerated and non-accelerated coordinate descent
coincide. However we note that this only applies
for matrices where $\mu$ is as large as it can be (i.e. $p/n$), that is
instances for which Gauss-Seidel is already converging at the optimal rate
(see \cite{gower15}, Lemma 4.2).

\begin{figure}[t!]
  \centering
  \begin{minipage}[t]{0.45\textwidth}
  \begin{center}
    \includegraphics[width=\columnwidth]{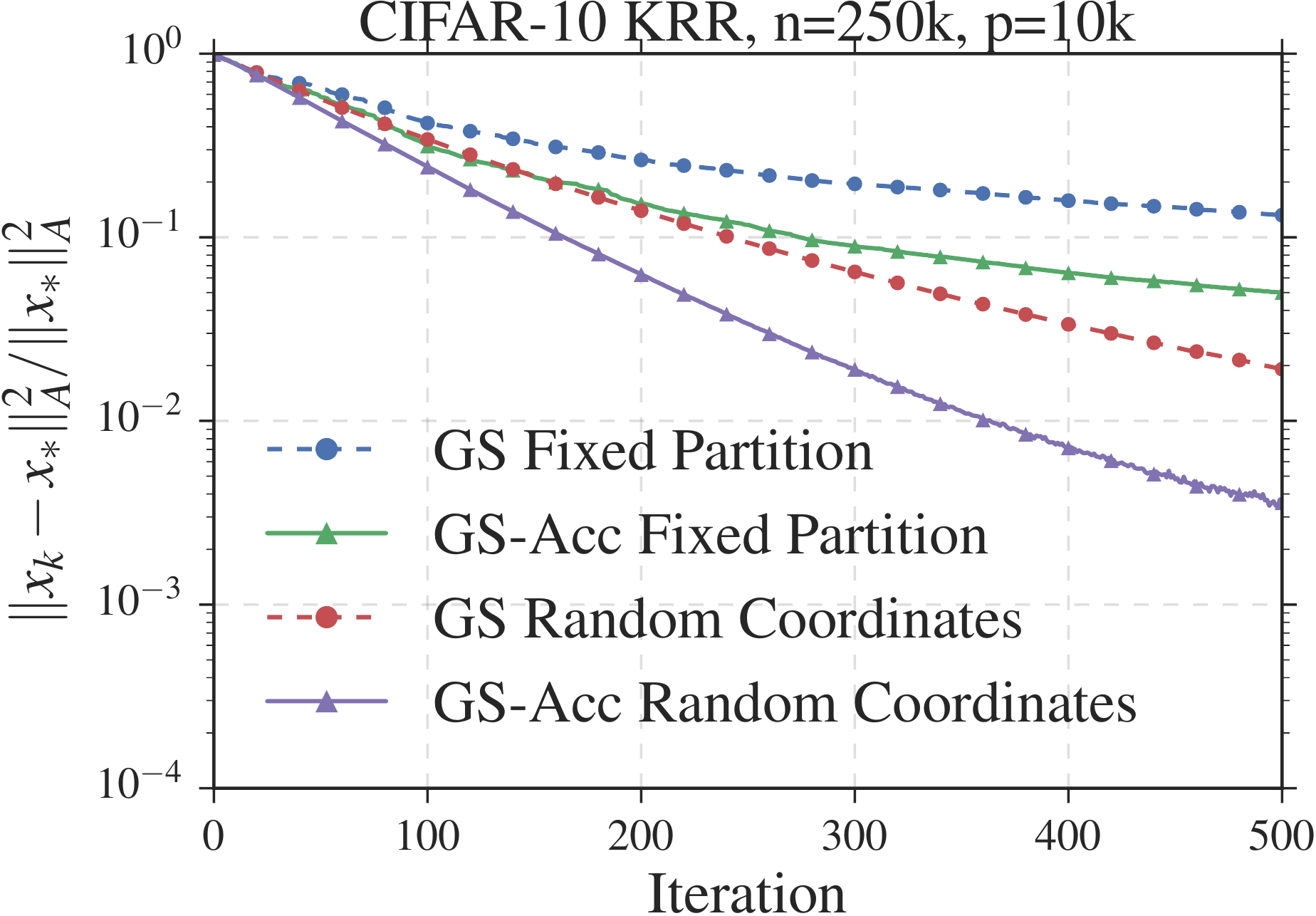}
  \end{center}
  \caption{Experiments comparing fixed partitions versus uniform random sampling for CIFAR-10
augmented matrix while running kernel ridge regression.
The matrix has $n=250000$ coordinates and we set block size to $p=10000$.
    }
  \label{fig:krr_parts_vs_coords}
  \end{minipage}
  \hspace{.06\textwidth}
  \begin{minipage}[t]{0.45\textwidth}
    \begin{center}
    \includegraphics[width=\columnwidth]{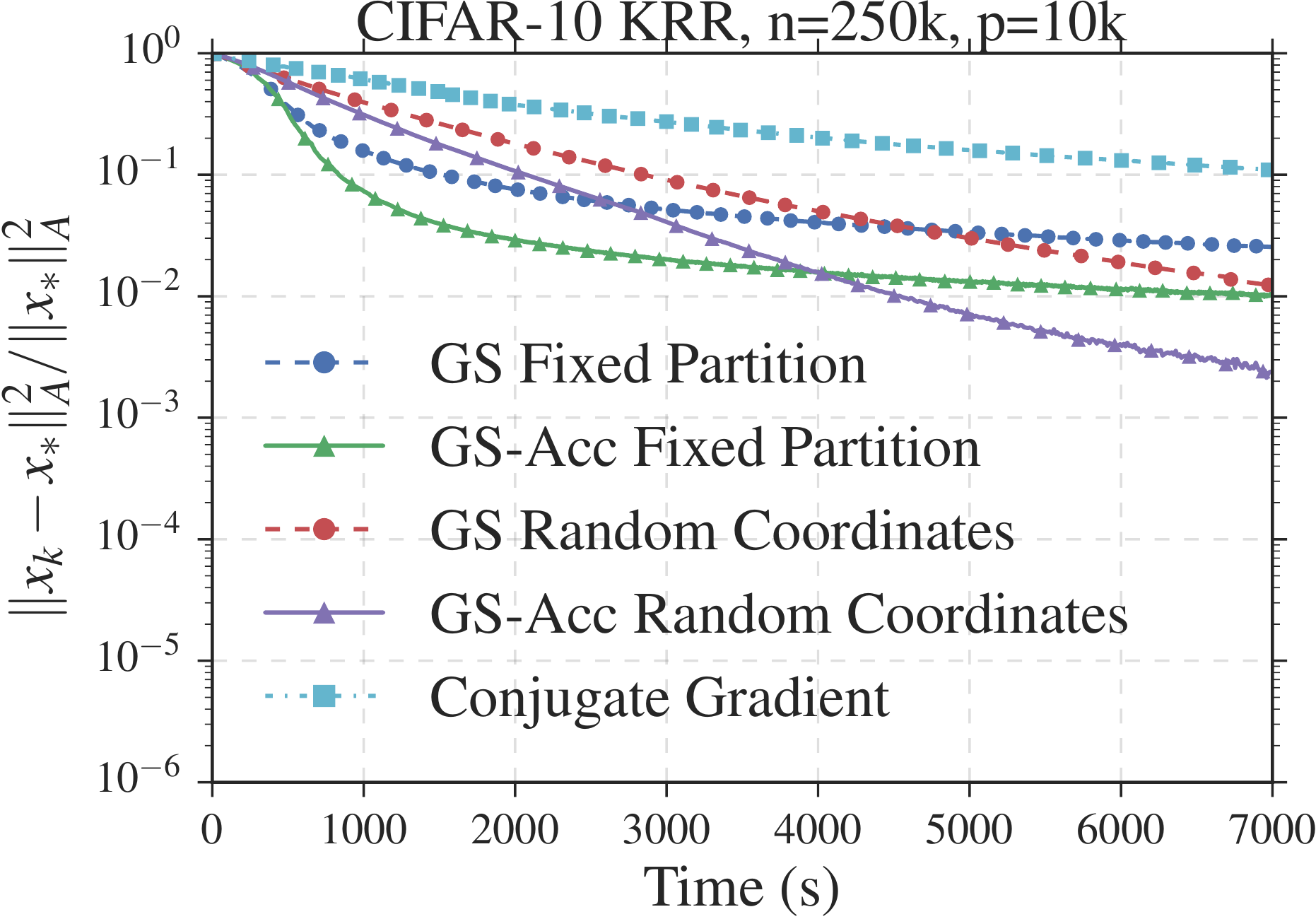}
    \end{center}
    \caption{Comparing conjugate gradient with accelerated and un-accelerated Gauss-Seidel
  methods for CIFAR-10 augmented matrix while running kernel ridge regression.
  The matrix has $n=250000$ coordinates and we set block size to $p=10000$.}
    \label{fig:krr_parts_vs_coords_vs_cg}
  \end{minipage}

\end{figure}

\subsection{Kernel ridge regression}
\label{sec:experiments:parts_vs_coords:krr}

We next evaluate how fixed partitioning and random coordinate sampling affects
the performance of Gauss-Seidel on large scale machine learning tasks.
We use the popular image classification dataset CIFAR-10 and evaluate a kernel ridge regression
(KRR) task with a Gaussian kernel.
Specifically, given a labeled dataset $\{(x_i, y_i)\}_{i=1}^{n}$,
we solve the linear system
$(K + \lambda I) \alpha = Y$
with $K_{ij} = \exp(-\gamma \twonorm{x_i - x_j}^2)$,
where $\lambda, \gamma > 0$ are tunable parameters
(see e.g.~\cite{scholkopf01} for background on KRR).
The key property of KRR is that the kernel matrix $K$ is positive semi-definite, and hence
Algorithm~\ref{alg:rand_acc_gs} applies.

For the CIFAR-10 dataset, we augment the dataset\footnote{Similar to
\url{https://github.com/akrizhevsky/cuda-convnet2}.} to include five
reflections, translations per-image and then apply standard
pre-processing steps used in image classification~\cite{coates12, sparks16}.
We finally apply a Gaussian kernel on our pre-processed images and
the resulting kernel matrix has $n=250000$ coordinates.

Results from running 500 iterations of random coordinate sampling and fixed partitioning algorithms
are shown in Figure~\ref{fig:krr_parts_vs_coords}. Comparing convergence across iterations, similar
to previous section, we see that un-accelerated Gauss-Seidel with random coordinate sampling is
better than accelerated Gauss-Seidel with fixed partitioning. However we also see that using
acceleration with random sampling can further improve the convergence rates, especially to achieve
errors of $10^{-3}$ or lower.

We also compare the convergence with respect to running time in
Figure~\ref{fig:krr_parts_vs_coords_vs_cg}. Fixed partitioning has better performance in practice
random access is expensive in multi-core systems. However, we see that this speedup in implementation
comes at a substantial cost in terms of convergence rate.
For example in the case of CIFAR-10, using fixed partitions leads to an error of $1.2 \times
10^{-2}$ after around 7000 seconds. In comparison
we see that random coordinate sampling achieves a similar error in around 4500 seconds and is thus
$1.5 \times$ faster. We also note that this speedup increases for lower error tolerances.

\subsection{Comparing Gauss-Seidel to Conjugate-Gradient}
We also compared Gauss-Seidel with random coordinate sampling to the classical
conjugate-gradient (CG) algorithm. CG is an important baseline to compare with, as
it is the de-facto standard iterative algorithm for solving linear systems
in the numerical analysis community.
While we report the results of CG without preconditioning, we remark that the
performance using a standard banded preconditioner was not any better.
However, for KRR specifically, there have been recent efforts~\cite{avron17,rudi17}
to develop better preconditioners, and we leave a more thorough comparison
for future work.
The results of our experiment are shown in Figure~\ref{fig:krr_parts_vs_coords_vs_cg}.
We note that Gauss-Seidel both with and without acceleration outperform CG.
As an example, we note that to reach error $10^{-1}$ on CIFAR-10, CG takes
roughly 7000 seconds, compared to less than 2000 seconds for
accelerated Gauss-Seidel, which is a $3.5 \times$ improvement.

To understand this performance difference, we recall that our
matrices $A$ are fully dense, and hence each iteration of CG takes $O(n^2)$.
On the other hand, each
iteration of both non-accelerated and accelerated Gauss-Seidel takes $O(np^2 +
p^3)$. Hence, as long as $p = O(n^{2/3})$, the time per iteration of
Gauss-Seidel is order-wise no worse than CG.  In terms of iteration complexity,
standard results state that CG takes at most $O(\sqrt{\kappa}
\log(1/\varepsilon))$ iterations to reach an $\varepsilon$ error solution,
where $\kappa$ denotes the condition number of $A$.  On the other hand,
Gauss-Seidel takes at most
$O(\frac{n}{p} \kappa_{\mathrm{eff}} \log(1/\varepsilon))$, where $\kappa_{\mathrm{eff}} =
\frac{\max_{1 \leq i \leq n} A_{ii}}{\lambda_{\min}(A)}$.  In the case of any
(normalized) kernel matrix associated with a translation-invariant kernel such
as the Gaussian kernel, we have $\max_{1 \leq i \leq n} A_{ii} = 1$, and hence
generally speaking $\kappa_{\mathrm{eff}}$ is much lower than $\kappa$.

\subsection{Kernel ridge regression on smaller datasets}
\label{sec:experiments:parts_vs_coords:krr:mnist}

In addition to using the large CIFAR-10 augmented dataset, we also tested our algorithms on the
smaller MNIST\footnote{\url{http://yann.lecun.com/exdb/mnist/}} dataset. To generate a kernel matrix, we
applied the Gaussian kernel on the raw MNIST pixels to generate a matrix $K$ with $n=60000$
rows and columns.

Results from running 500 iterations of random coordinate sampling and fixed partitioning algorithms
are shown in Figure~\ref{fig:krr_parts_vs_coords:mnist}. We plot the convergence rates both across time and
across iterations. Comparing convergence across iterations we see that random
coordinate sampling is essential to achieve errors of $10^{-4}$ or lower.
In terms of running time, similar to the CIFAR-10 experiment, we see that the benefits in fixed
partitioning of accessing coordinates faster comes at a cost in terms of convergence
rate, especially to achieve errors of $10^{-4}$ or lower.

\begin{figure}[h]
  \centering
  \begin{minipage}[t]{0.45\textwidth}
  \begin{center}
  \includegraphics[width=\columnwidth]{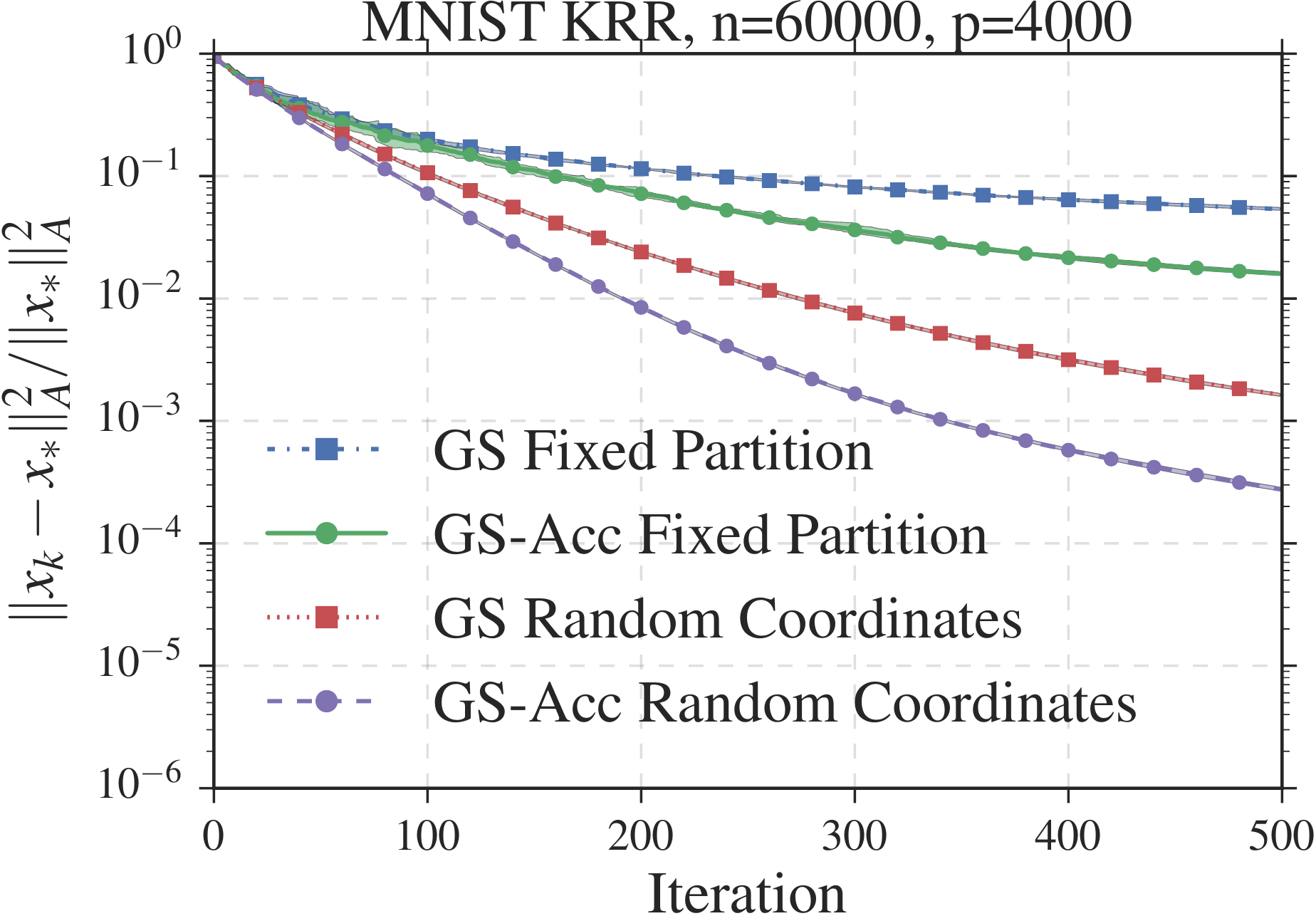}
  \end{center}
  \end{minipage}
  \hspace{.06\textwidth}
  \begin{minipage}[t]{0.45\textwidth}
  \begin{center}
  \includegraphics[width=\columnwidth]{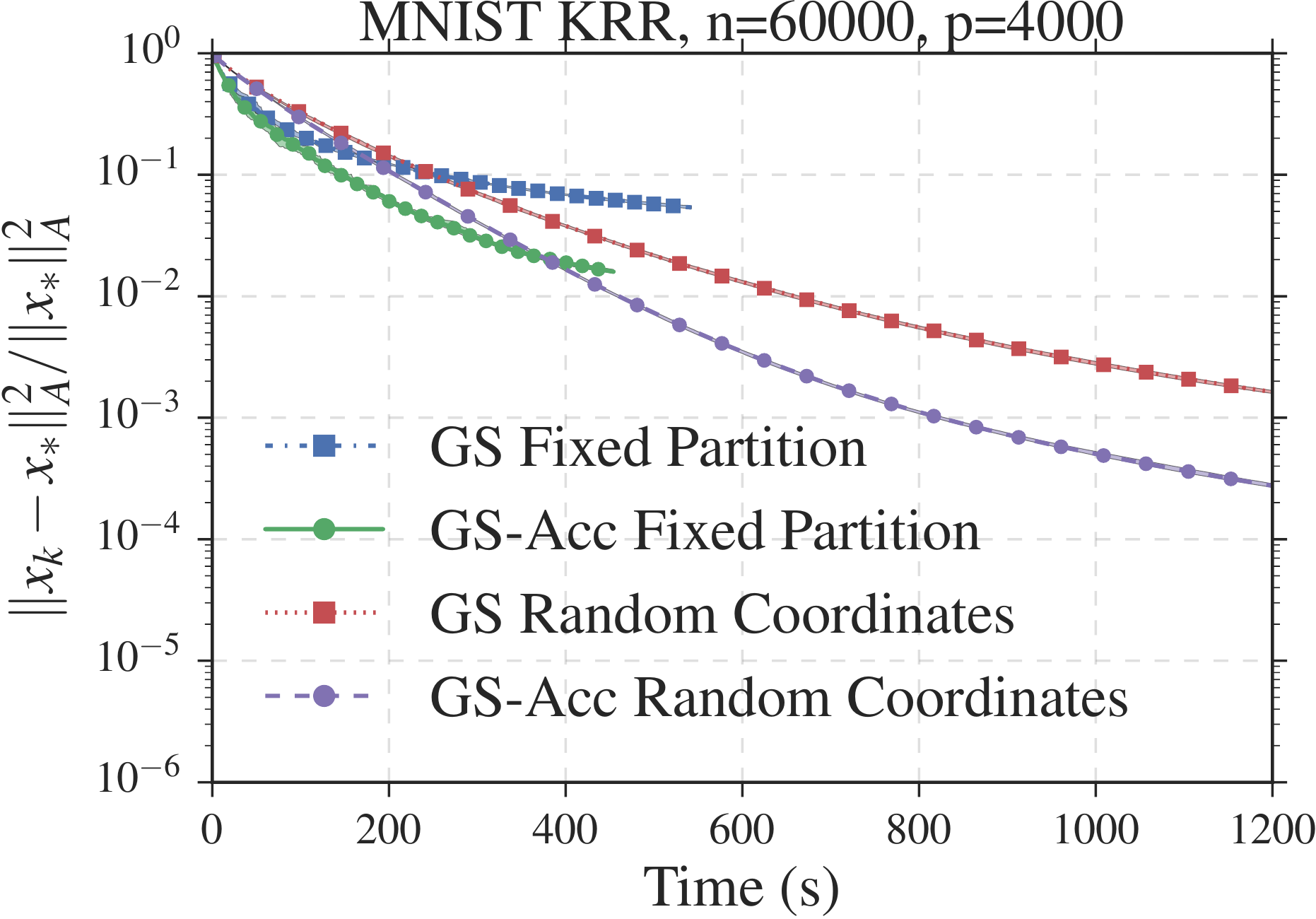}
  \end{center}
  \end{minipage}
  \caption{Experiments comparing fixed partitions versus uniform random sampling for MNIST
while running kernel ridge regression.
MNIST has $n=60000$ coordinates and we set block size to $p=4000$.}
    \label{fig:krr_parts_vs_coords:mnist}
\end{figure}

\subsection{Effect of block size}
We next analyze the importance of the block size $p$ for the accelerated Gauss-Seidel method. As the
values of $\mu$ and $\nu$ change for each setting of $p$, we use a smaller MNIST matrix for this
experiment. We apply a random feature transformation~\cite{rahimi07} to generate an $n \times d$
matrix $F$ with $d=5000$ features. We then use $A = F^\T F$ and $b = F^\T Y$ as inputs to the
algorithm. Figure~\ref{fig:block_size_effect} shows the wall clock time to
converge to $10^{-5}$ error as we vary the block size from $p=50$ to $p=1000$.

Increasing the block-size improves the amount of progress that is made per iteration but the time
taken per iteration increases as $O(p^3)$ (Line~\ref{alg:line:localupdate},
Algorithm~\ref{alg:rand_acc_gs}). However, using efficient BLAS-3 primitives usually affords a
speedup from systems techniques like cache blocking. We see the effects of this in
Figure~\ref{fig:block_size_effect} where using $p=500$ performs better than using $p=50$. We also
see that these benefits reduce for much larger block sizes and thus $p=1000$ is slower.

\begin{figure}[t!]
  \centering
  \begin{minipage}[t]{0.32\textwidth}
  \begin{center}
  \includegraphics[width=\columnwidth]{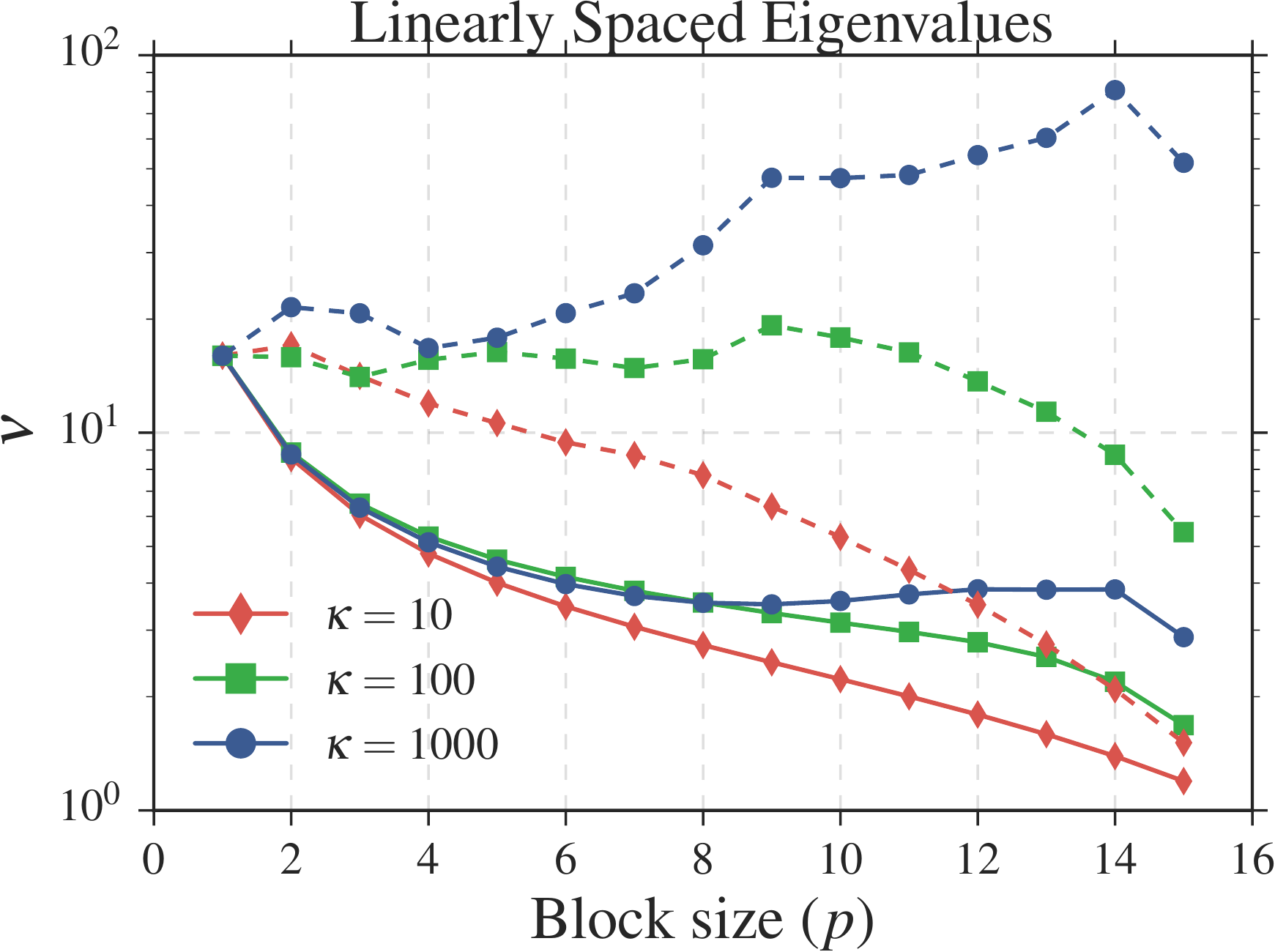}
  \end{center}
  \end{minipage}
  %\hspace{.06\textwidth}
    \begin{minipage}[t]{0.32\textwidth}
    \begin{center}
  \includegraphics[width=\columnwidth]{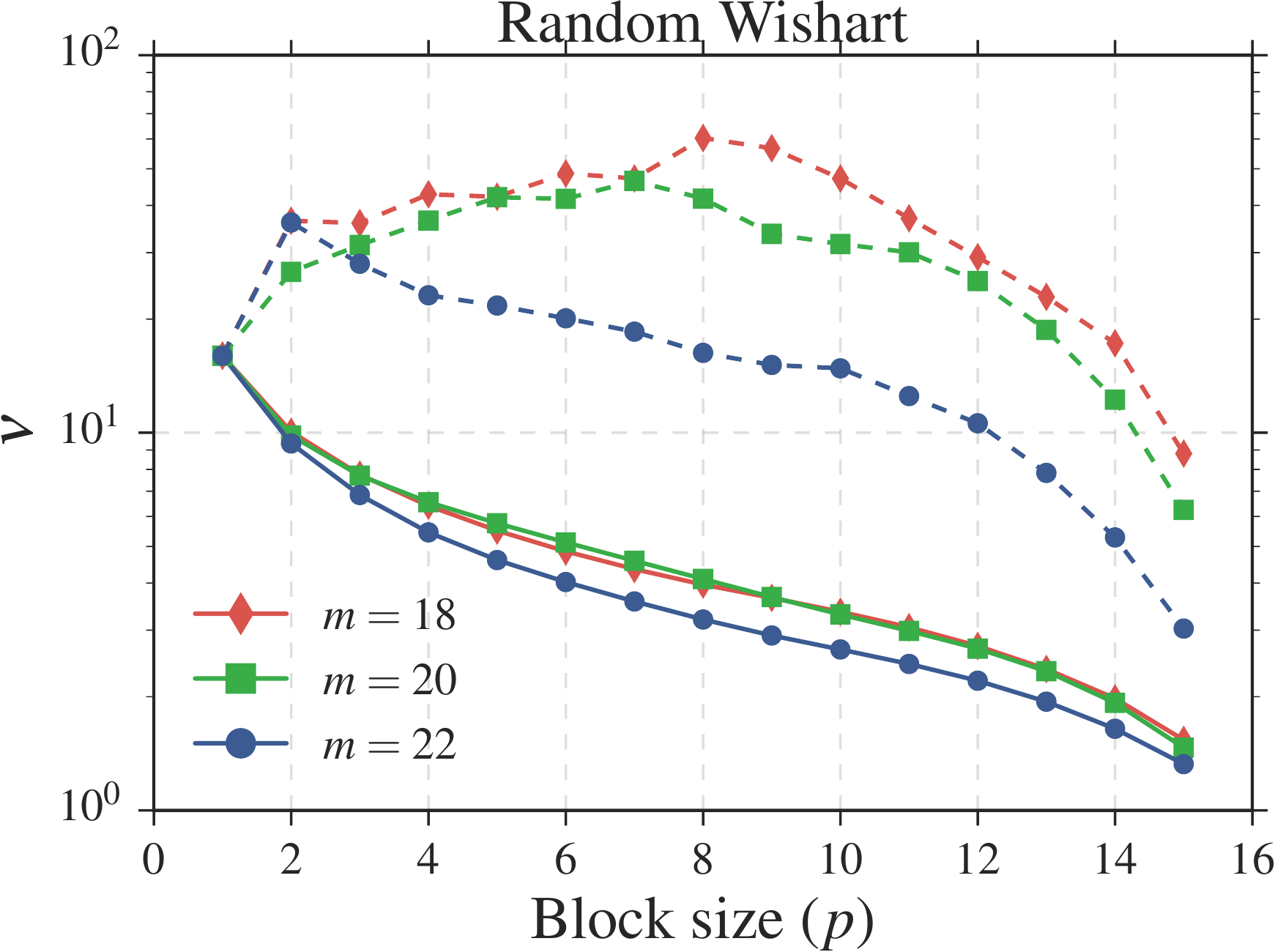}
  \end{center}
  \end{minipage}
  \begin{minipage}[t]{0.32\textwidth}
  \begin{center}
  \includegraphics[width=\columnwidth]{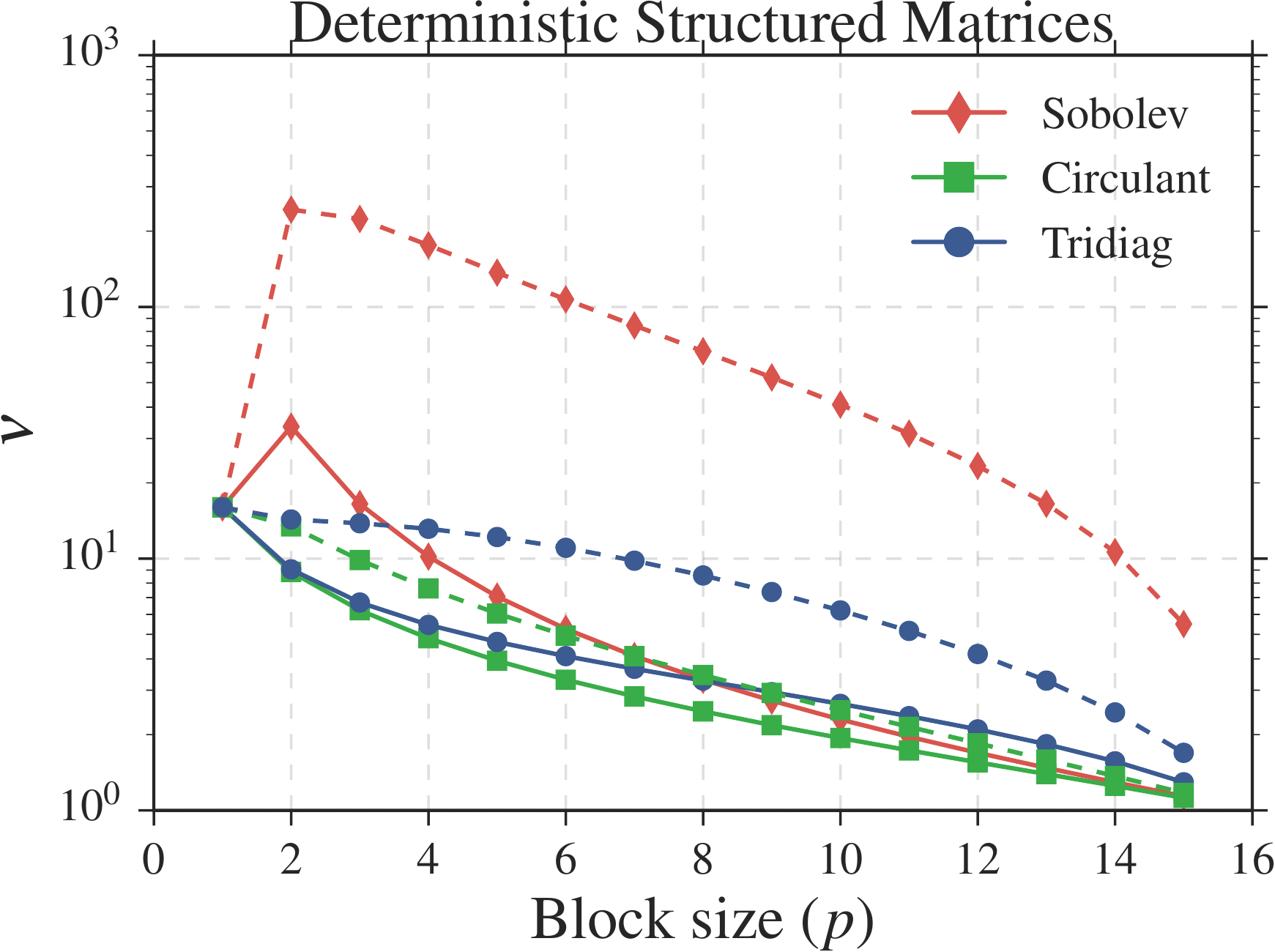}
  \end{center}
  \end{minipage}
    \caption{Comparison of the computed $\nu$ constant (solid lines) and $\nu$ bound
        from Theorem~\ref{thm:rand_acc_gs} (dotted lines) on random matrices
        with linearly spaced eigenvalues and random Wishart matrices. }
    \label{fig:nu_linspace_gaussian}
\end{figure}

\subsection{Computing the $\mu$ and $\nu$ constants}
\label{sec:experiments:mu_and_nu}

In our last experiment, we explicitly compute the $\mu$ and $\nu$ constants from
Theorem~\ref{thm:rand_acc_gs} for a few $16 \times 16$ positive definite matrices constructed as follows.

\noindent\textbf{Linearly spaced eigenvalues.}
We first draw $Q$ uniformly at random from $n \times n$ orthogonal matrices.
We then construct $A_i = Q \Sigma_i Q^\T$ for $i=1,2,3$, where
$\Sigma_1$ is \texttt{diag(linspace(1, 10, 16))}, $\Sigma_2$ is \texttt{diag(linspace(1, 100, 16))},
and $\Sigma_3$ is \texttt{diag(linspace(1, 1000, 16))}.

\noindent\textbf{Random Wishart.}
We first draw $B_i$ with iid $N(0, 1)$ entries, where
$B_i \in \R^{m_i \times 16}$ with $m_1 = 18$, $m_2 = 20$, and $m_3 = 22$.
We then set $A_i = B_i^\T B_i$.

\noindent\textbf{Sobolev kernel.}
We form the matrix $A_{ij} = \min(i, j)$ with $1 \leq i,j  \leq n$.
This corresponds to the gram matrix for the set of points
$x_1, ..., x_n \in \R$ with $x_i = i$ under the Sobolev kernel $\min(x,y)$.

\noindent\textbf{Circulant matrix.}
We let $A$ be a $16 \times 16$ instance of the family of circulant matrices
$A_n = F_n \diag(c_n) F_n^*$ where $F_n$ is the $n \times n$ unitary DFT matrix and
$c_n = (1, 1/2, ..., 1/(n/2+1), ..., 1/2, 1)$.
By construction this yields a real valued circulant matrix which is positive definite.

\noindent\textbf{Tridiagonal matrix.}
We let $A$ be a tridiagonal matrix with the diagonal value equal to one,
and the off diagonal value equal to $(\delta-a)/(2\cos(\pi n/(n+1)))$ for $\delta = 1/10$.
The matrix has a minimum eigenvalue of $\delta$.
\\

Figure~\ref{fig:nu_linspace_gaussian} shows the results of our computation
for the linearly spaced eigenvalues ensemble, the random Wishart ensemble and
the other deterministic structured matrices.
Alongside with the actual $\nu$ values, we plot the bound given for each instance
by Lemma~\ref{lemma:nu_upper_bound}. From the figures we see that our bound is quite close
to the computed value of $\nu$ for circulant matrices and for random matrices with linearly spaced
eigenvalues with small $\kappa$. We plan to extend our analysis to derive a tighter bound in the
future.

%% file: conclusion.tex
\section{Conclusion}
\label{sec:conclusion}

In this paper, we extended the accelerated block Gauss-Seidel algorithm beyond fixed partition
sampling.  Our analysis introduced a new data-dependent parameter $\nu$ which governs the speed-up
of acceleration. Specializing our theory to random coordinate
sampling, we derived an upper bound on $\nu$ which shows that well conditioned blocks are a
sufficient condition to ensure speedup.  Experimentally, we showed that 
random coordinate sampling is readily accelerated beyond what our bound suggests.

The most obvious question remains to derive a sharper bound on the $\nu$ constant from Theorem
\ref{thm:rand_acc_gs}. 
Another interesting question is whether or not
the iteration complexity of random coordinate sampling
is always bounded above by the iteration complexity with fixed coordinate sampling.

We also plan to study an implementation of accelerated
Gauss-Seidel in a distributed setting~\cite{tu16}.  The main challenges
here are in determining how to sample coordinates
without significant communication overheads, and to efficiently estimate
$\mu$ and $\nu$.  To do this, we wish to
explore other sampling schemes such as shuffling the coordinates at the
end of every epoch~\cite{recht2013parallel}.

%% file: notation.tex
\section{Preliminaries}

\paragraph{Notation.}
The notation is standard.
$[n] = \{1, 2, ..., n\}$ refers to the set of integers from $1$ to $n$,
and $2^{[n]}$ refers to the set of all subsets of $[n]$.
We let $\ind_n \in \R^{n}$ denote the vector of all ones.
Given a square matrix $M$ with real eigenvalues, we let $\lambda_{\max}(M)$ (resp. $\lambda_{\min}(M)$)
denote the maximum (resp. minimum) eigenvalue of $M$.
For two symmetric matrices $M, N$, the notation $M \succcurlyeq N$ (resp. $M \succ N$)
means that the matrix $M - N$ is positive semi-definite (resp. positive definite).
Every such $M \succ 0$ defines a real inner product space via
the inner product $\ip{x}{y}_{M} = x^\T M y$. We refer to its
induced norm as $\norm{x}_{M} = \sqrt{\ip{x}{x}_{M}}$. The standard
Euclidean inner product and norm will be denoted as $\ip{\cdot}{\cdot}$ and
$\twonorm{\cdot}$, respectively.
For an arbitrary matrix $M$, we let $M^{\dag}$ denote its Moore-Penrose pseudo-inverse
and $P_{M}$ the orthogonal projector onto the range of $M$, which we denote as
$\mathcal{R}(M)$.
When $M \succcurlyeq 0$, we let $M^{1/2}$ denote its unique Hermitian square root.
Finally, for a square $n \times n$ matrix $M$, $\diag(M)$ is the $n \times n$ diagonal matrix which
contains the diagonal elements of $M$.

\paragraph{Partitions on $[n]$.} In what follows,
unless stated otherwise, whenever we discuss a partition of $[n]$
we assume that the partition is given by $\bigcup_{i=1}^{n/p} J_i$, where
\begin{align*}
    J_1 = \{1, 2, ..., p\} \:, \:\: J_2 = \{p+1, p+2, ..., 2p\}, \:\: ... \:.
\end{align*}
This is without loss of generality because for any arbitrary equal sized partition of $[n]$,
there exists a permutation matrix $\Pi$ such that all our results apply by the change of variables
$A \gets \Pi^\T A \Pi$ and $b \gets \Pi^\T b$.

%% file: proofs_separation.tex
\section{Proofs for Separation Results (Section~\ref{sec:results:separation})}

\subsection{Expectation calculations (Propositions \ref{prop:simple_calcs_model2_fixed} and \ref{prop:simple_calcs_model2_random})}

Recall the family of $n \times n$ positive definite matrices $\mathscr{A}$ defined in \eqref{eq:model1} as
\begin{align}
    A_{\alpha,\beta} = \alpha I + \frac{\beta}{n} \ind_n\ind_n^\T \:, \;\; \alpha > 0, \alpha + \beta > 0 \:. \label{eq:appendix:general_family}
\end{align}
We first gather some elementary formulas. By the matrix inversion lemma,
\begin{align}
    A_{\alpha,\beta}^{-1} = \left(\alpha I + \frac{\beta}{n} \ind_n\ind_n^\T\right)^{-1} = \alpha^{-1}I - \frac{\beta/n}{\alpha(\alpha+\beta)} \ind_n\ind_n^\T \:.
\end{align}
Furthermore, let $S \in \R^{n \times p}$ be any column selector matrix with
no duplicate columns. We have again by the matrix inversion lemma
\begin{align}
    (S^\T A_{\alpha,\beta} S)^{-1} = \left(\alpha I + \frac{\beta}{n} \ind_p\ind_p^\T\right)^{-1} = \alpha^{-1}I - \frac{\beta/n}{\alpha(\alpha+\beta p/n)} \ind_p\ind_p^\T \:.
\end{align}
The fact that the right hand side is independent of $S$ is the key property
which makes our calculations possible.
Indeed, we have that
\begin{align}
    S(S^\T A_{\alpha,\beta} S)^{-1} S^\T = \alpha^{-1} SS^\T - \frac{\beta/n}{\alpha(\alpha+\beta p/n)} S\ind_p\ind_p^\T S^\T \:. \label{eq:appendix:term1}
\end{align}
With these formulas in hand, our next proposition gathers
calculations for the case when $S$ represents
uniformly choosing $p$ columns without replacement.

\begin{proposition}
\label{prop:simple_calculations}
Consider the family of $n \times n$ positive definite
matrices $\{A_{\alpha,\beta}\}$ from \eqref{eq:appendix:general_family}.
Fix any integer $p$ such that $1 < p < n$.
Let $S \in \R^{n \times p}$ denote a random column selector matrix
where each column of $S$ is chosen uniformly at random without replacement
from $\{e_1, ..., e_n\}$.
For any $A_{\alpha,\beta}$,
\begin{align}
    \E[ S(S^\T A_{\alpha,\beta} S)^{-1} S^\T A_{\alpha,\beta} ] &= p\frac{(n-1)\alpha + (p-1)\beta}{(n-1)(n\alpha + p\beta)}I + \frac{(n-p)p\beta}{n(n-1)(n\alpha + p\beta)} \ind_n\ind_n^\T \:, \label{eq:appendix:GA} \\
    \E[ S(S^\T A_{\alpha,\beta} S)^{-1} S^\T G_{\alpha,\beta}^{-1} S(S^\T A_{\alpha,\beta} S)^{-1} S^\T ] &= \left(\frac{1}{\alpha} - \frac{(n-p)^2\beta}{(n-1)((n-1)\alpha + (p-1)\beta)(n\alpha + p\beta)} \right) I \nonumber \\
    &\qquad+ \frac{(p-1)\beta(n\alpha(1-2n) + np(\alpha-\beta) + p\beta)}{(n-1)n\alpha((n-1)\alpha + (p-1)\beta)(n\alpha + p\beta)} \ind_n\ind_n^\T \label{eq:appendix:variance_term} \:.
\end{align}
Above, $G_{\alpha,\beta} = \E[ S(S^\T A_{\alpha,\beta} S)^{-1} S^\T ]$.
\end{proposition}
\begin{proof}
First, we have the following elementary expectation calculations,
\begin{align}
    \E[ SS^\T ] &= \frac{p}{n} I \:, \label{eq:appendix:ex1} \\
    \E[ S \ind_p\ind_p^\T S^\T ] &= \frac{p}{n}\left(1 - \frac{p-1}{n-1}\right) I + \frac{p}{n}\left(\frac{p-1}{n-1}\right) \ind_n\ind_n^\T \:, \label{eq:appendix:ex2} \\
    \E[ SS^\T \ind_n\ind_p^\T S^\T ] &= \E[ S \ind_p\ind_n^\T SS^\T ] = \E[ SS^\T \ind_n\ind_n^\T SS^\T ] = \E[ S \ind_p \ind_p^\T S^\T ] \:, \label{eq:appendix:ex3} \\
    \E[ S \ind_p\ind_p^\T S^\T \ind_n\ind_n^\T S \ind_p\ind_p^\T S^\T ] &= \frac{p^3}{n}\left(1-\frac{p-1}{n-1}\right)I + \frac{p^3}{n} \left(\frac{p-1}{n-1}\right) \ind_n\ind_n^\T \label{eq:appendix:ex4} \:.
\end{align}
To compute $G_{\alpha,\beta}$, we simply
plug \eqref{eq:appendix:ex1} and \eqref{eq:appendix:ex2} into \eqref{eq:appendix:term1}. After simplification,
    \begin{align*}
        G_{\alpha,\beta} = \E[ S(S^\T A_{\alpha,\beta} S)^{-1} S^\T ] = \frac{p}{\alpha n}\left(1 - \frac{\beta/n}{\alpha + \beta p/n}\left(1 - \frac{p-1}{n-1}\right)\right) I - \frac{p}{n}\frac{p-1}{n-1} \frac{\beta/n}{\alpha(\alpha + \beta p/n)} \ind_n\ind_n^\T \:.
    \end{align*}
From this formula for $G_{\alpha,\beta}$, \eqref{eq:appendix:GA} follows immediately.

Our next goal is to compute $\E[ S(S^\T A_{\alpha,\beta} S)^{-1} S^\T G_{\alpha,\beta}^{-1} S(S^\T A_{\alpha,\beta} S)^{-1} S^\T ]$. To do this, we first invert $G_{\alpha,\beta}$.
Applying the matrix inversion lemma, we can write down a formula for the inverse of $G_{\alpha,\beta}$,
    \begin{align}
        G_{\alpha,\beta}^{-1} = \underbrace{\frac{(n-1)\alpha(n\alpha + p\beta)}{(n-1)p\alpha + (p-1)p\beta}}_{\gamma} I + \underbrace{\frac{(p-1)\beta(n\alpha + p\beta)}{np((n-1)\alpha + (p-1)\beta)}}_{\eta}\ind_n\ind_n^\T \:. \label{eq:appendix:ginv}
    \end{align}
Next, we note for any $r, q$,
using the properties that $S^\T S = I$, $\ind_n^\T S \ind_p = p$, and $\ind_p^\T \ind_p = p$, we have that
\begin{align*}
    &(r SS^\T + q S\ind_p\ind_p^\T S^\T) (\gamma I + \eta \ind_n\ind_n^\T)(r SS^\T + q S\ind_p\ind_p^\T S^\T) \\
    &\qquad= \gamma r^2 SS^\T + 2r\gamma q S \ind_p\ind_p^\T S^\T + \eta r^2 SS^\T \ind_n\ind_n^\T SS^\T \\
    &\qquad\qquad+ p r \eta q (SS^\T \ind_n\ind_p^\T S^\T + S \ind_p\ind_n^\T SS^\T ) + pq^2\gamma S \ind_p\ind_p^\T S^\T \\
    &\qquad\qquad+\eta q^2 S \ind_p\ind_p^\T S^\T \ind_n\ind_n^\T S \ind_p\ind_p^\T S^\T \:.
\end{align*}
Taking expectations of both sides of the above equation and using the formulas in
\eqref{eq:appendix:ex1}, \eqref{eq:appendix:ex2}, \eqref{eq:appendix:ex3},
and \eqref{eq:appendix:ex4},
\begin{align*}
    &\E[ (r SS^\T + q S\ind_p\ind_p^\T S^\T) (\gamma I + \eta \ind_n\ind_n^\T)(r SS^\T + q S\ind_p\ind_p^\T S^\T) ] \\
    &\qquad= \frac{p(p(n-p)q^2 + 2(n-p)qr + (n-1)r^2)\gamma + p(n-p)(pq+r)^2\eta}{n(n-1)} I \\
    &\qquad\qquad+ \frac{p(p-1)( q(pq+2r)\gamma + (pq+r)^2\eta)}{n(n-1)} \ind_n\ind_n^\T \:.
\end{align*}
We now set $r = \alpha^{-1}$, $q = -\frac{\beta/n}{\alpha(\alpha+\beta p/n)}$,
and $\gamma, \eta$ from \eqref{eq:appendix:ginv} to reach the desired formula
for \eqref{eq:appendix:variance_term}.
\end{proof}
Proposition~\ref{prop:simple_calcs_model2_random}
follows immediately from Proposition~\ref{prop:simple_calculations} by plugging
in $\alpha = 1$ into \eqref{eq:appendix:GA}.
We next consider how \eqref{eq:appendix:term1} behaves under
a fixed partition of $\{1, ..., n\}$.
Recall our assumption on partitions:
$n = pk$ for some integer $k \geq 1$,
and we sequentially partition $\{1, ..., n\}$ into $k$ partitions of size $p$,
i.e. $J_1 = \{1, ..., p\}$, $J_2 = \{p+1, ..., 2p\}$, and so on.
Define $S_1, ..., S_k \in \R^{n \times p}$ such that $S_i$ is the column
selector matrix for the partition $J_i$, and $S$ uniformly chooses $S_i$ with
probability $1/k$.

\begin{proposition}
\label{prop:appendix:fixed_partition_calc}
Consider the family of $n \times n$ positive definite
matrices $\{A_{\alpha,\beta}\}$ from \eqref{eq:appendix:general_family}, and let
$n$, $p$, and $S$ be described as in the preceding paragraph. We have that
    \begin{align}
        \E[ S(S^\T A_{\alpha,\beta} S)^{-1} S^\T A_{\alpha,\beta} ] = \frac{p}{n} I + \frac{p\beta}{n^2\alpha + np\beta}\ind_n\ind_n^\T - \frac{p\beta}{n^2\alpha + np\beta}\mathrm{blkdiag}(\underbrace{\ind_p\ind_p^\T, ..., \ind_p\ind_p^\T}_{k \text{ times}}) \:. \label{eq:appendix:GS_partition}
    \end{align}
\end{proposition}
\begin{proof}
Once again, the expectation calculations are
\begin{align*}
    \E[ SS^\T ] = \frac{p}{n} I, \;\; \E[ S \ind_p\ind_p^\T S^\T ] = \frac{p}{n}\mathrm{blkdiag}(\underbrace{\ind_p\ind_p^\T, ..., \ind_p\ind_p^\T}_{k \textrm{ times}}) \:.
\end{align*}
Therefore,
\begin{align*}
    \E[ S(S^\T A_{\alpha,\beta} S)^{-1} S^\T ] = \frac{p}{\alpha n} I - \frac{p}{n} \frac{\beta/n}{\alpha(\alpha+\beta p/n)} \mathrm{blkdiag}(\ind_p\ind_p^\T, ..., \ind_p\ind_p^\T) \:.
\end{align*}
Furthermore,
\begin{align*}
    \mathrm{blkdiag}(\ind_p\ind_p^\T, ..., \ind_p\ind_p^\T) \ind_n\ind_n^\T = \ind_n\ind_n^\T\mathrm{blkdiag}(\ind_p\ind_p^\T, ..., \ind_p\ind_p^\T) = p \ind_n\ind_n^\T \:,
\end{align*}
Hence, the formula for
$\E[ S(S^\T A_{\alpha,\beta} S)^{-1} S^\T A_{\alpha,\beta} ]$ follows.
\end{proof}

We now make the following observation. Let $Q_1, ..., Q_k$ be any partition
of $\{1, ..., n\}$ into $k$ partitions of size $p$.
Let $\E_{S \sim Q_i}$ denote expectation with respect to $S$ uniformly chosen as column selectors among $Q_1, ..., Q_k$,
and let $\E_{S \sim J_i}$ denote expectation with respect to the $S$
in the setting of Proposition~\ref{prop:appendix:fixed_partition_calc}. It is not hard to see
there exists a permutation matrix $\Pi$ such that
\begin{align*}
    \Pi^\T \E_{S \sim Q_i}[ S(S^\T A_{\alpha,\beta} S)^{-1} S^\T ] \Pi  = \E_{S \sim J_i}[ S(S^\T A_{\alpha,\beta} S)^{-1} S^\T ] \:.
\end{align*}
Using this permutation matrix $\Pi$,
\begin{align*}
    \lambda_{\min}( \E_{S \sim Q_i}[ P_{A^{1/2}_{\alpha,\beta} S} ] ) &= \lambda_{\min} ( \E_{S \sim Q_i}[ S(S^\T A_{\alpha,\beta} S)^{-1} S^\T ] A_{\alpha,\beta} ) \\
    &= \lambda_{\min} ( \E_{S \sim Q_i}[ S(S^\T A_{\alpha,\beta} S)^{-1} S^\T ] \Pi A_{\alpha,\beta} \Pi^\T ) \\
    &= \lambda_{\min} ( \Pi^\T \E_{S \sim Q_i}[ S(S^\T A_{\alpha,\beta} S)^{-1} S^\T ] \Pi A_{\alpha,\beta} ) \\
    &= \lambda_{\min} (\E_{S \sim J_i}[ S(S^\T A_{\alpha,\beta} S)^{-1} S^\T ] A_{\alpha,\beta}) \\
    &= \lambda_{\min} ( \E_{S \sim J_i}[ P_{A^{1/2}_{\alpha,\beta} S} ] ) \:.
\end{align*}
Above, the second equality holds because $A_{\alpha,\beta}$ is invariant under a similarity transform by any permutation matrix.
Therefore, Proposition~\ref{prop:appendix:fixed_partition_calc} yields the
$\mu_{\mathrm{part}}$ value for every partition $Q_1, ..., Q_k$.
The claim of Proposition~\ref{prop:simple_calcs_model2_fixed} now follows
by substituting $\alpha=1$ into \eqref{eq:appendix:GS_partition}.

\subsection{Proof of Proposition~\ref{prop:lower_bound_gs}}
Define $e_k = x_k - x_*$, $H_k = S_k(S_k^\T A S_k)^{\dag} S_k^\T$ and $G = \E[ H_k ]$.
From the update rule \eqref{eq:rand_gs},
\begin{align*}
    e_{k+1} = (I - H_k A) e_k \Longrightarrow A^{1/2} e_{k+1} = (I - A^{1/2} H_k A^{1/2}) A^{1/2} e_k \:.
\end{align*}
Taking and iterating expectations,
\begin{align*}
    \E[A^{1/2} e_{k+1}] = (I - A^{1/2} G A^{1/2}) \E[A^{1/2} e_k] \:.
\end{align*}
Unrolling this recursion yields for all $k \geq 0$,
\begin{align*}
    \E[A^{1/2} e_k] = (I - A^{1/2} G A^{1/2})^k A^{1/2} e_0 \:.
\end{align*}
Choose $A^{1/2} e_0 = v$, where $v$ is an eigenvector of $I - A^{1/2} G A^{1/2}$ with eigenvalue
$\lambda_{\max}(I - A^{1/2} G A^{1/2}) = 1 - \lambda_{\min}(G A) = 1 - \mu$.
Now by Jensen's inequality,
\begin{align*}
    \E[\norm{e_k}_{A}] = \E[\twonorm{A^{1/2} e_k}] \geq \twonorm{\E[A^{1/2} e_k]} = (1-\mu)^{k} \norm{e_0}_{A} \:.
\end{align*}
This establishes the claim.

%% file: proofs_convergence.tex
\section{Proofs for Convergence Results (Section~\ref{sec:results:convergence})}
\label{sec:appendix:framework}

We now state our main structural result for accelerated coordinate descent.
Let $\Pr$ be a probability measure on $\Omega = \mathcal{S}^{n \times n} \times
\R_+ \times \R_+$, with $\mathcal{S}^{n \times n}$ denoting $n \times n$
positive semi-definite matrices and $\R_+$ denoting positive reals.  Write
$\omega \in \Omega$ as the tuple $\omega = (H, \Gamma, \gamma)$, and let $\E$ denote
expectation with respect to $\Pr$.
Suppose that $G = \E[ \frac{1}{\gamma} H ]$ exists and is positive definite.

Now suppose that $f : \R^{n} \longrightarrow \R$ is a differentiable and strongly convex
function, and put $f_* = \min_{x} f(x)$, with $x_*$ attaining the minimum value.
Suppose that
$f$ is both $\mu$-strongly convex and has $L$-Lipschitz gradients with respect to
the $G^{-1}$ norm.
This means that for all $x,y \in \R^{n}$, we have
\begin{subequations}
\begin{align}
    f(y) &\geq f(x) + \ip{\nabla f(x)}{y - x} + \frac{\mu}{2} \norm{y-x}_{G^{-1}}^2 \:, \label{eq:mu_strong_cvx} \\
    f(y) &\leq f(x) + \ip{\nabla f(x)}{y - x} + \frac{L}{2} \norm{y-x}_{G^{-1}}^2 \label{eq:L_lipschitz} \:.
\end{align}
\end{subequations}

We now define a random sequence as follows.
Let $\omega_0 = (H_0, \Gamma_0, \gamma_0), \omega_1 = (H_1, \Gamma_1, \gamma_1), ...$
be independent realizations from $\Pr$.
Starting from $y_0=z_0=x_0$ with $x_0$ fixed, consider the sequence
$\{(x_k,y_k,z_k)\}_{k \geq 0}$ defined by the recurrence
\begin{subequations}
    \label{eq:new_update_rule}
    \begin{align}
        \tau(x_{k+1} - z_k) &= y_k - x_{k+1} \:, \label{eq:acc:coupling} \\
        y_{k+1} &= x_{k+1} - \frac{1}{\Gamma_k} H_k \nabla f(x_{k+1}) \:, \label{eq:acc:gradient_step} \\
        z_{k+1} - z_k &= \tau \left( x_{k+1} - z_k - \frac{1}{\mu \gamma_k} H_k \nabla f(x_{k+1}) \right) \:. \label{eq:acc:zseq}
    \end{align}
\end{subequations}
It is easily verified that $(x, y, z) = (x_*, x_*, x_*)$ is a fixed point
of the aforementioned dynamical system.
Our goal for now is to describe conditions on $f$, $\mu$, and $\tau$ such that
the sequence of updates  \eqref{eq:acc:coupling}, \eqref{eq:acc:gradient_step}, and \eqref{eq:acc:zseq}
converges to this fixed point. As described in Wilson et al.~\cite{wilson16}, our main strategy for
proving convergence will be to introduce the following Lyapunov function
\begin{align}
    V_k = f(y_k) - f_* + \frac{\mu}{2} \norm{z_k - x_*}_{G^{-1}}^2 \:, \label{eq:lyapunov_function}
\end{align}
and show that $V_k$ decreases along every trajectory.
We let $\E_k$ denote the expectation conditioned on $\mathcal{F}_k = \sigma(\omega_0, \omega_1, ..., \omega_{k-1})$.
Observe that $x_{k+1}$ is $\mathcal{F}_k$-measurable, a fact we will use repeatedly throughout our
calculations.
With the preceding definitions in place, we state and prove our main structural theorem.
\begin{theorem}
\label{thm:main_structural}
(Generalization of Theorem~\ref{thm:general_simplified}.)
Let $f$ and $G$ be as defined above,
with $f$ satisfying $\mu$-strongly convexity and $L$-Lipschitz gradients with respect to the $\norm{\cdot}_{G^{-1}}$ norm,
as defined in \eqref{eq:mu_strong_cvx} and \eqref{eq:L_lipschitz}.
Suppose that for all fixed $x \in \R^{n}$, we have that the following holds for almost every $\omega \in \Omega$,
\begin{align}
    f(\Phi(x;\omega)) \leq f(x) - \frac{1}{2\Gamma} \norm{\nabla f(x)}^2_{H} \:, \:\: \Phi(x;\omega) = x - \frac{1}{\Gamma} H \nabla f(x) \:. \label{eq:gradient_step_progress}
\end{align}
Furthermore, suppose that $\nu > 0$ satisfies
\begin{align}
    \E\left[ \frac{1}{\gamma^2} H G^{-1} H \right] \preccurlyeq \nu \E \left[ \frac{1}{\gamma^2} H \right] \:. \label{eq:nu_parameter}
\end{align}
Then as long as we set $\tau > 0$ such that $\tau$ satisfies for almost every $\omega \in \Omega$,
\begin{align}
    \tau \leq \frac{\gamma}{\sqrt{\Gamma}} \sqrt{\frac{\mu}{\nu}} \:, \:\: \tau \leq \sqrt{\frac{\mu}{L}} \:, \label{eq:tau_requirements}
\end{align}
we have that $V_k$ defined in \eqref{eq:lyapunov_function}
satisfies for all $k \geq 0$,
\begin{align}
    \E_k [V_{k+1}] \leq (1-\tau) V_k \:. \label{eq:lyapunov_recurrence}
\end{align}
\end{theorem}
\begin{proof}
First, recall the following two point equality valid for
any vectors $a, b, c \in V$ in a real inner product space $V$,
\begin{align}
    \norm{a-b}^2_{V} - \norm{c-b}^2_{V} = \norm{a-c}^2_{V} + 2\ip{a-c}{c-b}_{V} \:. \label{eq:two_point_eq}
\end{align}
Now we can proceed with our analysis,
\begin{subequations}
\begin{align}
    V_{k+1} - V_k \,\,\,&\eqrefstackrel{eq:two_point_eq}{=}\,\,\, f(y_{k+1}) - f(y_k) - \mu\ip{z_{k+1}-z_k}{x_* - z_k}_{G^{-1}} + \frac{\mu}{2}\norm{z_{k+1}-z_k}^2_{G^{-1}} \notag\\
    &= \,\,\, f(y_{k+1}) - f(x_{k+1}) + f(x_{k+1}) - f(y_k) - \mu\ip{z_{k+1}-z_k}{x_* - z_k}_{G^{-1}} + \frac{\mu}{2}\norm{z_{k+1}-z_k}^2_{G^{-1}} \notag\\
    &\eqrefstackrel{eq:mu_strong_cvx}{\leq} \,\,\, f(y_{k+1}) - f(x_{k+1}) + \ip{\nabla f(x_{k+1})}{x_{k+1}-y_k} - \frac{\mu}{2} \norm{x_{k+1} - y_k}^2_{G^{-1}}  \notag\\
        &\qquad- \mu\ip{z_{k+1}-z_k}{x_* - z_k}_{G^{-1}} + \frac{\mu}{2}\norm{z_{k+1}-z_k}^2_{G^{-1}} \label{eq:note:main:1} \\
    &\eqrefstackrel{eq:acc:zseq}{=} \,\,\, f(y_{k+1}) - f(x_{k+1}) + \ip{\nabla f(x_{k+1})}{x_{k+1}-y_k} - \frac{\mu}{2} \norm{x_{k+1} - y_k}^2_{G^{-1}}  \notag\\
        &\qquad+ \tau\ip{ \frac{1}{\gamma_k} H_k \nabla f(x_{k+1}) - \mu(x_{k+1} - z_k) }{x_* - z_k}_{G^{-1}} + \frac{\mu}{2}\norm{z_{k+1}-z_k}^2_{G^{-1}} \label{eq:note:main:2} \\
    &= \,\,\, f(y_{k+1}) - f(x_{k+1}) + \ip{\nabla f(x_{k+1})}{x_{k+1}-y_k} - \frac{\mu}{2} \norm{x_{k+1} - y_k}^2_{G^{-1}}  \notag\\
        &\qquad+ \tau\ip{ \frac{1}{\gamma_k} H_k \nabla f(x_{k+1}) }{x_* - x_{k+1}}_{G^{-1}}  + \tau \ip{ \frac{1}{\gamma_k} H_k \nabla f(x_{k+1}) }{x_{k+1} - z_k}_{G^{-1}} \notag\\
        &\qquad- \tau\mu\ip{ x_{k+1} - z_k }{x_* - z_k}_{G^{-1}} + \frac{\mu}{2}\norm{z_{k+1}-z_k}^2_{G^{-1}} \notag\\
    &\eqrefstackrel{eq:acc:zseq}{=} \,\,\, f(y_{k+1}) - f(x_{k+1}) + \ip{\nabla f(x_{k+1})}{x_{k+1}-y_k} - \frac{\mu}{2} \norm{x_{k+1} - y_k}^2_{G^{-1}}  \notag\\
        &\qquad+ \tau\ip{ \frac{1}{\gamma_k} H_k \nabla f(x_{k+1}) }{x_* - x_{k+1}}_{G^{-1}}  + \tau \ip{ \frac{1}{\gamma_k} H_k \nabla f(x_{k+1}) }{x_{k+1} - z_k}_{G^{-1}} \notag \\
        &\qquad- \tau\mu\ip{ x_{k+1} - z_k }{x_* - z_k}_{G^{-1}} + \frac{\mu}{2} \norm{\tau (x_{k+1}-z_k)}^2_{G^{-1}} + \frac{\tau^2}{2\mu \gamma_k^2} \norm{H_k \nabla f(x_{k+1})}^2_{G^{-1}} \notag \\
        &\qquad- \tau\ip{x_{k+1}-z_k}{ \tau \frac{1}{\gamma_k} H_k \nabla f(x_{k+1})}_{G^{-1}} \label{eq:note:main:3} \\
    &\eqrefstackrel{eq:gradient_step_progress}{\leq} \,\,\, - \frac{1}{2\Gamma_k} \norm{\nabla f(x_{k+1})}^2_{H_k} + \ip{\nabla f(x_{k+1})}{x_{k+1}-y_k} - \frac{\mu}{2} \norm{x_{k+1} - y_k}^2_{G^{-1}}  \notag\\
        &\qquad+ \tau\ip{ \frac{1}{\gamma_k} H_k \nabla f(x_{k+1}) }{x_* - x_{k+1}}_{G^{-1}}  + \tau \ip{ \frac{1}{\gamma_k} H_k \nabla f(x_{k+1}) }{x_{k+1} - z_k}_{G^{-1}} \notag \\
        &\qquad- \tau\mu\ip{ x_{k+1} - z_k }{x_* - z_k}_{G^{-1}} + \frac{\mu}{2} \norm{\tau (x_{k+1}-z_k)}^2_{G^{-1}} + \frac{\tau^2}{2\mu \gamma_k^2} \norm{H_k \nabla f(x_{k+1})}^2_{G^{-1}} \notag \\
        &\qquad- \tau\ip{x_{k+1}-z_k}{ \tau \frac{1}{\gamma_k} H_k \nabla f(x_{k+1})}_{G^{-1}} \label{eq:note:main:4} \:.
\end{align}
\end{subequations}
Above, \eqref{eq:note:main:1} follows from $\mu$-strong convexity,
\eqref{eq:note:main:2} and \eqref{eq:note:main:3} both use the definition of the update sequence
given in \eqref{eq:new_update_rule},
and \eqref{eq:note:main:4} follows using the gradient inequality assumption
\eqref{eq:gradient_step_progress}.
Now letting $x \in \R^{n}$ be fixed, we observe that
\begin{align}
    \E\left[ \frac{\tau^2}{2\mu \gamma^2} \nabla f(x)^\T H G^{-1} H \nabla f(x) - \frac{1}{2\Gamma} \norm{\nabla f(x)}_{H}^2 \right] \,\,\,&\eqrefstackrel{eq:nu_parameter}{\leq}\,\,\, \E\left[\left( \frac{\tau^2 \nu}{2\mu \gamma^2} - \frac{1}{2\Gamma}\right) \norm{\nabla f(x)}_{H}^2 \right] \notag\\
    &\eqrefstackrel{eq:tau_requirements}{\leq} 0 \:. \label{eq:nu_inequality}
\end{align}
The first inequality uses the assumption on $\nu$,
and the second inequality uses the requirement that
$\tau \leq \frac{\gamma}{\sqrt{\Gamma}} \sqrt{\frac{\mu}{\nu}}$.
Now taking expectations with respect to $\E_k$,
\begin{subequations}
\begin{align}
    \E_k [V_{k+1}] - V_k &\leq \,\,\, \E_k\left[ \frac{\tau^2}{2\mu \gamma_k^2} \nabla f(x_{k+1})^\T H_k G^{-1} H_k \nabla f(x_{k+1}) - \frac{1}{2\Gamma_k} \norm{\nabla f(x_{k+1})}_{H_k}^2 \right] \notag\\
        &\qquad+ \ip{\nabla f(x_{k+1})}{x_{k+1}-y_k} - \frac{\mu}{2} \norm{x_{k+1} - y_k}^2_{G^{-1}}  \notag\\
        &\qquad+ \tau\ip{\nabla f(x_{k+1})}{x_* - x_{k+1}} + \tau \ip{\nabla f(x_{k+1})}{x_{k+1} - z_k} - \tau\mu\ip{ x_{k+1} - z_k }{x_* - z_k}_{G^{-1}} \notag\\
        &\qquad+\frac{\mu}{2} \norm{\tau (x_{k+1}-z_k)}^2_{G^{-1}} - \tau\ip{x_{k+1}-z_k}{\tau \nabla f(x_{k+1})} \notag\\
    &\eqrefstackrel{eq:nu_inequality}{\leq} \,\,\, \ip{\nabla f(x_{k+1})}{x_{k+1}-y_k} - \frac{\mu}{2} \norm{x_{k+1} - y_k}^2_{G^{-1}} + \tau\ip{\nabla f(x_{k+1})}{x_* - x_{k+1}} \notag\\
        &\qquad+ \tau \ip{\nabla f(x_{k+1})}{x_{k+1} - z_k} - \tau\mu\ip{ x_{k+1} - z_k }{x_* - z_k}_{G^{-1}} \notag\\
        &\qquad+\frac{\mu}{2} \norm{\tau (x_{k+1}-z_k)}^2_{G^{-1}} - \tau\ip{x_{k+1}-z_k}{\tau \nabla f(x_{k+1})} \notag\\
    &\eqrefstackrel{eq:mu_strong_cvx}{\leq} \,\,\, -\tau\left( f(x_{k+1}) - f_* + \frac{\mu}{2} \norm{x_{k+1} - x_*}^2_{G^{-1}} \right) + \ip{\nabla f(x_{k+1})}{x_{k+1}-y_k} - \frac{\mu}{2} \norm{x_{k+1} - y_k}^2_{G^{-1}} \notag\\
        &\qquad+ \tau \ip{\nabla f(x_{k+1})}{x_{k+1} - z_k} - \tau\mu\ip{ x_{k+1} - z_k }{x_* - z_k}_{G^{-1}} \notag\\
        &\qquad+\frac{\mu}{2} \norm{\tau (x_{k+1}-z_k)}^2_{G^{-1}} - \tau\ip{x_{k+1}-z_k}{\tau \nabla f(x_{k+1})} \label{eq:note:main:5}\\
    &\eqrefstackrel{eq:acc:coupling}{=} \,\,\, - \tau\left( f(x_{k+1}) - f_* + \frac{\mu}{2} \norm{x_{k+1} - x_*}^2_{G^{-1}} \right) - \frac{\mu}{2} \norm{x_{k+1} - y_k}^2_{G^{-1}} \notag\\
        &\qquad- \tau\mu\ip{ x_{k+1} - z_k }{x_* - z_k}_{G^{-1}} \notag\\
        &\qquad+\frac{\mu}{2} \norm{\tau (x_{k+1}-z_k)}^2_{G^{-1}} - \tau\ip{y_k - x_{k+1}}{\nabla f(x_{k+1})} \label{eq:note:main:6}\\
    &\eqrefstackrel{eq:L_lipschitz}{\leq} \,\,\, - \tau\left( f(x_{k+1}) - f_* + \frac{\mu}{2} \norm{x_{k+1} - x_*}^2_{G^{-1}} \right) - \frac{\mu}{2} \norm{x_{k+1} - y_k}^2_{G^{-1}} \notag\\
        &\qquad- \tau\mu\ip{ x_{k+1} - z_k }{x_* - z_k}_{G^{-1}} \notag\\
        &\qquad+\frac{\mu}{2} \norm{\tau (x_{k+1}-z_k)}^2_{G^{-1}} + \tau(f(x_{k+1}) - f(y_k)) + \frac{\tau L}{2} \norm{y_k - x_{k+1}}^2_{G^{-1}} \label{eq:note:main:7}\\
    &\eqrefstackrel{eq:two_point_eq}{=} \,\,\, -\tau\left( f(x_{k+1}) - f_* + \frac{\mu}{2} \norm{x_{k+1} - z_k}^2_{G^{-1}} + \frac{\mu}{2}\norm{z_k - x_*}^2_{G^{-1}} + \mu\ip{x_{k+1} - z_k}{z_k - x_*}_{G^{-1}} \right) \notag\\
        &\qquad- \frac{\mu}{2} \norm{x_{k+1} - y_k}^2_{G^{-1}} - \tau\mu\ip{ x_{k+1} - z_k }{x_* - z_k}_{G^{-1}} \notag\\
        &\qquad+\frac{\mu}{2} \norm{\tau (x_{k+1}-z_k)}^2_{G^{-1}} + \tau(f(x_{k+1}) - f(y_k)) + \frac{\tau L}{2} \norm{y_k - x_{k+1}}^2_{G^{-1}} \label{eq:note:main:8}\\
    &\eqrefstackrel{eq:lyapunov_function}{=} \,\,\, - \tau V_k - \frac{\mu}{2} \norm{x_{k+1} - y_k}^2_{G^{-1}} - \frac{\tau \mu}{2} \norm{x_{k+1} - z_k}^2_{G^{-1}} \notag\\
        &\qquad+ \frac{\mu}{2} \norm{\tau (x_{k+1}-z_k)}^2_{G^{-1}} + \frac{\tau L}{2} \norm{y_k - x_{k+1}}^2_{G^{-1}} \notag\\
    &\eqrefstackrel{eq:acc:coupling}{=} \,\,\, - \tau V_k + \left( \frac{\tau L}{2} - \frac{ \mu}{2\tau} \right) \norm{y_k - x_{k+1}}^2_{G^{-1}} \label{eq:note:main:9} \\
    &\eqrefstackrel{eq:tau_requirements}{\leq} \,\,\, -\tau V_k \notag \:.
\end{align}
\end{subequations}
Above, \eqref{eq:note:main:5} follows from $\mu$-strong convexity,
\eqref{eq:note:main:6} and \eqref{eq:note:main:9} both use the definition of the
sequence \eqref{eq:new_update_rule},
\eqref{eq:note:main:7} follows from $L$-Lipschitz gradients,
\eqref{eq:note:main:8} uses the two-point inequality \eqref{eq:two_point_eq},
and the last inequality follows from the assumption of $\tau \leq \sqrt{\frac{\mu}{L}}$.
The claim \eqref{eq:lyapunov_recurrence} now follows by re-arrangement.
\end{proof}

\subsection{Proof of Theorem~\ref{thm:rand_acc_gs}}
\label{sec:appendix:proof_main_gs_result}

Next, we describe how to recover Theorem~\ref{thm:rand_acc_gs} from Theorem~\ref{thm:main_structural}.
We do this by applying Theorem~\ref{thm:main_structural} to the function
$f(x) = \frac{1}{2} x^\T A x - x^\T b$.

The first step in applying Theorem~\ref{thm:main_structural} is to construct
a probability measure on $\mathcal{S}^{n \times n} \times \R_+ \times \R_+$ for
which the randomness of the updates is drawn from.
We already have a distribution on $\mathcal{S}^{n \times n}$
from setting of Theorem~\ref{thm:rand_acc_gs} via the random matrix $H$. We trivially augment
this distribution by considering the random variable $(H, 1, 1) \in \Omega$.
By setting $\Gamma=\gamma=1$, the sequence \eqref{eq:acc:coupling},
\eqref{eq:acc:gradient_step}, \eqref{eq:acc:zseq} reduces to that
of Algorithm~\ref{alg:rand_acc_gs}.
Furthermore, the requirement on the $\nu$ parameter from \eqref{eq:nu_parameter}
simplifies to the requirement listed in \eqref{eq:gs_variance_term}.
This holds by the following equivalences which are valid
since conjugation by $G$ (which is assumed to be positive definite) preserves the semi-definite ordering,
\begin{align}
    \lambda_{\max}\left( \E\left[ (G^{-1/2} H G^{-1/2})^2 \right] \right) \leq \nu
    &\Longleftrightarrow  \E\left[ (G^{-1/2} H G^{-1/2})^2 \right] \preccurlyeq \nu I \notag\\
    &\Longleftrightarrow  \E\left[ G^{-1/2} H G^{-1} H G^{-1/2} \right] \preccurlyeq \nu I \notag\\
    &\Longleftrightarrow  \E\left[ H G^{-1} H \right] \preccurlyeq \nu G \label{eq:nu_param_equiv}\:.
\end{align}

It remains to check the gradient inequality \eqref{eq:gradient_step_progress}
and compute the strong convexity and Lipschitz parameters.  These computations
fall directly from the calculations made in Theorem 1 of \cite{qu15}, but we
replicate them here for completeness.

To check the gradient inequality \eqref{eq:gradient_step_progress},
because $f$ is a quadratic function, its second order Taylor
expansion is exact. Hence for almost every $\omega \in \Omega$,
\begin{align*}
    f(\Phi(x;\omega)) &= f(x) - \ip{\nabla f(x)}{H \nabla f(x)} + \frac{1}{2} \nabla f(x)^\T H A H \nabla f(x) \\
    &= f(x) - \ip{\nabla f(x)}{H \nabla f(x)} + \frac{1}{2} \nabla f(x)^\T S (S^\T A S)^{\dag} S^\T A S (S^\T A S)^{\dag} S^\T \nabla f(x) \\
    &=f(x) - \ip{\nabla f(x)}{H \nabla f(x)} + \frac{1}{2} \nabla f(x)^\T S (S^\T A S)^{\dag} S^\T \nabla f(x) \\
    &= f(x) - \frac{1}{2} \nabla f(x)^\T H \nabla f(x) \:.
\end{align*}
Hence the inequality \eqref{eq:gradient_step_progress} holds with equality.

We next compute the strong convexity and Lipschitz gradient parameters.
We first show that $f$ is $\lambda_{\min}(\E[ P_{A^{1/2} S}])$-strongly convex
with respect to the $\norm{\cdot}_{G^{-1}}$ norm.
This
follows since for any $x, y \in \R^{n}$, using the assumption that $G$ is positive
definite,
\begin{align*}
    f(y) &= f(x) + \ip{\nabla f(x)}{y-x} + \frac{1}{2} (y-x)^\T A (y-x) \\
         &= f(x) + \ip{\nabla f(x)}{y-x} + \frac{1}{2} (y-x)^\T G^{-1/2} G^{1/2} A G^{1/2} G^{-1/2} (y - x) \\
         &\geq f(x) + \ip{\nabla f(x)}{y-x} + \frac{\lambda_{\min}(A^{1/2} G A^{1/2})}{2} \norm{y-x}_{G^{-1}}^2 \:.
\end{align*}
The strong convexity bound now follows since
\begin{align*}
    A^{1/2} G A^{1/2} = A^{1/2} \E[H] A^{1/2} = \E[A^{1/2} S(S^\T A S)^{\dag} S^\T A^{1/2}] = \E[P_{A^{1/2} S}] \:.
\end{align*}
An nearly identical argument shows that $f$ is $\lambda_{\max}( \E[P_{A^{1/2} S}] )$-strongly convex
with respect to the $\norm{\cdot}_{G^{-1}}$ norm. Since the eigenvalues of projector matrices
are bounded by 1, we have that $f$ is 1-Lipschitz
with respect to the $\norm{\cdot}_{G^{-1}}$ norm.
This calculation shows that the requirement on $\tau$ from \eqref{eq:tau_requirements}
simplifies to $\tau \leq \sqrt{\frac{\mu}{\nu}}$, since $L=1$ and $\nu \geq 1$ by
Proposition~\ref{prop:nu_lower_bound} which we state and prove later.

At this point, Theorem~\ref{thm:main_structural} yields that
$\E[ V_k ] \leq (1-\tau)^k V_0$.
To recover the final claim \eqref{eq:rand_acc_gs_rate},
recall that $f(y_k) - f_* = \frac{1}{2} \norm{y_k - x_*}^2_{A}$.
Furthermore, $\mu G^{-1} \preccurlyeq A$, since
\begin{align*}
  \mu \leq \lambda_{\min}(A^{1/2} G A^{1/2}) &\Longleftrightarrow \mu \leq \lambda_{\min}( G^{1/2} A G^{1/2} ) \\
  &\Longleftrightarrow \mu I \preccurlyeq G^{1/2} A G^{1/2} \\
  &\Longleftrightarrow \mu G^{-1} \preccurlyeq A \:.
\end{align*}
Hence, we can upper bound $V_0$ as follows
\begin{align*}
  V_0 &= f(y_0) - f_* + \frac{\mu}{2} \norm{ z_0 - x_* }^2_{G^{-1}}
  = \frac{1}{2}\norm{y_0 - x_*}^2_{A} + \frac{\mu}{2} \norm{ z_0 - x_* }^2_{G^{-1}} \\
  &\leq \frac{1}{2}\norm{y_0 - x_*}^2_{A} + \frac{1}{2} \norm{ z_0 - x_* }^2_{A}
  = \norm{x_0 - x_*}^2_{A} \:.
\end{align*}
On the other hand, we have that $\frac{1}{2} \norm{y_k - x_*}^2_{A} \leq V_k$.
Putting the inequalities together,
\begin{align*}
  \frac{1}{\sqrt{2}} \E[ \norm{y_k - x_*}_{A} ] \leq \sqrt{\E[ \frac{1}{2} \norm{y_k - x_*}^2_{A} ]} \leq \sqrt{ \E[ V_k ] } \leq \sqrt{(1-\tau)^k V_0} \leq (1-\tau)^{k/2} \norm{x_0 - x_*}^2_{A} \:,
\end{align*}
where the first inequality holds by Jensen's inequality.
The claimed inequality \eqref{eq:rand_acc_gs_rate} now follows.

\subsection{Proof of Proposition~\ref{prop:lower_bound_acc_gs}}

We first state and prove an elementary linear algebra fact
which we will use below in our calculations.
\begin{proposition}
\label{prop:diag_eigenvalues}
Let $A,B,C,D$ be $n \times n$ diagonal matrices, and define
$M = \begin{bmatrix} A & B \\ C & D \end{bmatrix}$.
The eigenvalues of $M$ are given by the union of the eigenvalues
of the $2 \times 2$ matrices
  \begin{align*}
    \begin{bmatrix} A_i & B_i \\ C_i & D_i \end{bmatrix} \:, \:\: i=1, ..., n \:,
  \end{align*}
where $A_i,B_i,C_i,D_i$ denote the $i$-th diagonal entry of $A,B,C,D$ respectively.
\end{proposition}
\begin{proof}
For every $s \in \C$ we have that the matrices $-C$ and $sI - D$ are diagonal and hence commute.
Applying the corresponding formula for a block matrix determinant under this assumption,
    \begin{align*}
        0 &= \det\begin{bmatrix} sI - A & - B \\ -C & sI - D \end{bmatrix} = \det( (sI - A )(sI - D) - BC ) \\
          &= \prod_{i=1}^{n} ((s - A_i)(s - D_i) - B_i C_i) = \prod_{i=1}^{n} \det\begin{bmatrix} s - A_i & - B_i \\ - C_i & s - D_i \end{bmatrix} \:.
    \end{align*}
\end{proof}

Now we proceed with the proof of Proposition~\ref{prop:lower_bound_acc_gs}.
Define $e_k = \begin{bmatrix} y_k - x_* \\ z_k - x_* \end{bmatrix}$.
It is easy to see from the definition of Algorithm~\ref{alg:rand_acc_gs} that
$\{e_k\}$ satisfies the recurrence
\begin{align*}
    e_{k+1} = \frac{1}{1+\tau} \begin{bmatrix}
            I - H_k A & \tau (I - H_k A) \\
            \tau(I - \frac{1}{\mu} H_k A) & I - \frac{\tau^2}{\mu} H_k A
    \end{bmatrix} e_k \:.
\end{align*}
Hence,
\begin{align*}
    &\begin{bmatrix}
        A^{1/2} & 0 \\
        0 & \mu^{1/2} G^{-1/2}
    \end{bmatrix} e_{k+1} \\
    &\qquad=\frac{1}{1+\tau}\begin{bmatrix}
        A^{1/2} & 0 \\
        0 & \mu^{1/2} G^{-1/2}
    \end{bmatrix}
    \begin{bmatrix}
            I - H_k A & \tau (I - H_k A) \\
            \tau(I - \frac{1}{\mu} H_k A) & I - \frac{\tau^2}{\mu} H_k A
    \end{bmatrix} e_k \\
    &\qquad=\frac{1}{1+\tau}
    \begin{bmatrix}
        A^{1/2} - A^{1/2} H_k A & \tau (A^{1/2} - A^{1/2} H_k A) \\
        \mu^{1/2} \tau G^{-1/2} (I - \frac{1}{\mu} H_k A) & \mu^{1/2} G^{-1/2} (I - \frac{\tau^2}{\mu} H_k A)
    \end{bmatrix} e_k \\
    &\qquad=\frac{1}{1+\tau}
    \begin{bmatrix}
        I - A^{1/2} H_k A^{1/2} & \mu^{-1/2} \tau (A^{1/2} - A^{1/2} H_k A) G^{1/2} \\
        \mu^{1/2} \tau G^{-1/2} (I - \frac{1}{\mu} H_k A) A^{-1/2} &  G^{-1/2} (I - \frac{\tau^2}{\mu} H_k A) G^{1/2}
    \end{bmatrix}\begin{bmatrix}
        A^{1/2} & 0 \\
        0 & \mu^{1/2} G^{-1/2}
    \end{bmatrix} e_k \:.
\end{align*}
Define $P = \begin{bmatrix}
        A & 0 \\
        0 & \mu G^{-1}
\end{bmatrix}$.
By taking and iterating expectations,
\begin{align*}
    \E[ P^{1/2} e_{k+1} ] = \frac{1}{1+\tau} \begin{bmatrix}
        I - A^{1/2} G A^{1/2} & \mu^{-1/2} \tau (A^{1/2} G^{1/2} - A^{1/2} G A G^{1/2} ) \\
        \mu^{1/2} \tau ( G^{-1/2} A^{-1/2}  - \frac{1}{\mu} G^{1/2} A^{1/2} ) & I - \frac{\tau^2}{\mu} G^{1/2} A G^{1/2}
    \end{bmatrix}  \E[ P^{1/2} e_k ] \:.
\end{align*}
Denote the matrix $Q = A^{1/2} G^{1/2}$. Unrolling the recurrence above yields that
\begin{align*}
    \E[ P^{1/2} e_{k} ] = R^k P^{1/2} e_0 \:, \:\: R = \frac{1}{1+\tau} \begin{bmatrix}
        I - QQ^\T & \mu^{-1/2} \tau (Q - QQ^\T Q) \\
        \mu^{1/2} \tau ( Q^{-1}  - \frac{1}{\mu} Q^\T ) & I - \frac{\tau^2}{\mu} Q^\T Q
    \end{bmatrix} \:.
\end{align*}
Write the SVD of $Q$ as $Q = U \Sigma V^\T$. Both $U$ and $V$ are $n \times n$ orthonormal matrices.
It is easy to see that $R^k$ is given by
\begin{align}
    R^k = \frac{1}{(1+\tau)^k}\begin{bmatrix} U & 0 \\ 0 & V \end{bmatrix}
    \begin{bmatrix}
        I -  \Sigma^2 & \mu^{-1/2} \tau (\Sigma - \Sigma^3) \\
        \mu^{1/2} \tau (\Sigma^{-1} - \frac{1}{\mu}\Sigma) & I - \frac{\tau^2}{\mu} \Sigma^2
        \end{bmatrix}^k \begin{bmatrix} U^\T & 0 \\ 0 & V^\T \end{bmatrix} \:. \label{eq:recurrence_matrix_rk}
\end{align}
Suppose we choose $P^{1/2} e_0$ to be a right singular vector of $R^k$ corresponding to the
maximum singular value $\sigma_{\max}(R^k)$.
Then we have that
    \begin{align*}
        \E[\twonorm{P^{1/2} e_k}] \geq \twonorm{\E[P^{1/2} e_k]} = \twonorm{R^k P^{1/2} e_0} = \sigma_{\max}(R^k) \twonorm{P^{1/2} e_0} \geq \rho(R^k) \twonorm{P^{1/2} e_0} \:,
    \end{align*}
where $\rho(\cdot)$ denotes the spectral radius.
The first inequality is Jensen's inequality, and the second inequality uses
the fact that the spectral radius is bounded above by any matrix norm.
The eigenvalues of $R^k$ are the $k$-th power of the eigenvalues of $R$
which, using the similarity transform \eqref{eq:recurrence_matrix_rk} along with Proposition~\ref{prop:diag_eigenvalues},
are given by the eigenvalues of the $2 \times 2$ matrices $R_i$ defined as
    \begin{align*}
        R_i = \frac{1}{1+\tau}\begin{bmatrix}
        1 -  \sigma_i^2 & \mu^{-1/2} \tau (\sigma_i - \sigma_i^3) \\
        \mu^{1/2} \tau (\sigma_i^{-1} - \frac{1}{\mu}\sigma_i) & 1 - \frac{\tau^2}{\mu} \sigma_i^2
        \end{bmatrix} \:, \:\: \sigma_i = \Sigma_{ii} \:, \:\: i = 1, ..., n \:.
    \end{align*}
On the other hand, since the entries in $\Sigma$ are given by the
eigenvalues of $A^{1/2} G^{1/2} G^{1/2} A^{1/2} = \E[ P_{A^{1/2} S} ]$,
there exists an $i$ such that $\sigma_i = \sqrt{\mu}$. This $R_i$ is upper triangular,
and hence its eigenvalues can be read off the diagonal. This shows that $\frac{1 - \tau^2}{1+\tau} = 1-\tau$ is an eigenvalue of $R$, and hence $(1-\tau)^k$ is an eigenvalue of $R^k$.
But this means that $(1-\tau)^k \leq \rho(R^k)$.
Hence, we have shown that
\begin{align*}
    \E[ \twonorm{P^{1/2} e_k} ] \geq (1-\tau)^k \twonorm{P^{1/2} e_0} \:.
\end{align*}
The desired claim now follows from
    \begin{align*}
        \twonorm{P^{1/2} e_k} &= \sqrt{\norm{y_k - x_*}_{A}^2 + \mu\norm{z_k - x_*}_{G^{-1}}^2} \\
        &\leq \sqrt{\norm{y_k - x_*}_{A}^2 + \norm{z_k - x_*}_{A}^2} \leq \norm{y_k - x_*}_{A} + \norm{z_k - x_*}_{A} \:,
    \end{align*}
where the first inequality holds since $\mu G^{-1} \preccurlyeq A$
and the second inequality holds since $\sqrt{a+b} \leq \sqrt{a}+\sqrt{b}$ for non-negative $a,b$.

\section{Recovering the ACDM Result from Nesterov and Stich~\cite{nesterov16}}
\label{sec:appendix:recovering_acdm}

We next show how to recover
Theorem 1 of Nesterov and Stich~\cite{nesterov16} using Theorem~\ref{thm:main_structural},
in the case of $\alpha=1$.
A nearly identical argument can also be used to recover the result of Allen-Zhu et al.~\cite{allenzhu16} under the
strongly convex setting in the case of $\beta=0$.
Our argument proceeds in two steps. First, we prove a convergence result for a
simplified accelerated coordinate descent method which we introduce in
Algorithm~\ref{alg:rand_acc_cd}.
Then, we describe how a minor tweak to ACDM shows the equivalence between ACDM
and Algorithm~\ref{alg:rand_acc_cd}.

Before we proceed, we first describe the setting of Theorem 1.
Let $f : \R^{n} \longrightarrow \R$ be a twice differentiable strongly convex function
with Lipschitz gradients.
Let $J_1, ..., J_m$ denote a partition of $\{1, ..., n\}$ into $m$ partitions.
Without loss of generality, we can assume that the partitions are in order,
i.e. $J_1 = \{1, ..., n_1\}$, $J_2 = \{n_1 + 1, ..., n_2\}$, and so on.
This is without loss of generality since we can always consider
the function $g(x) = f(\Pi x)$ for a suitable permutation matrix $\Pi$.
Let $B_1, ..., B_m$ be fixed positive definite matrices such that
$B_i \in \R^{\abs{J_i} \times \abs{J_i}}$.
Set $H_i = S_i B_i^{-1} S_i^\T$,
where $S_i \in \R^{n \times \abs{J_i}}$ is the column selector matrix associated to partition $J_i$,
and define $L_i = \sup_{x \in \R^{n}} \lambda_{\max}(B_i^{-1/2} S_i^\T \nabla^2 f(x) S_i B_i^{-1/2})$
for $i=1,...,m$. Furthermore, define $p_i = \frac{\sqrt{L_i}}{\sum_{j=1}^{m} \sqrt{L_j}}$.

\subsection{Proof of convergence of a simplified accelerated coordinate descent method}

Now consider the following accelerated randomized coordinate descent algorithm in Algorithm~\ref{alg:rand_acc_cd}.
\begin{algorithm}[h!]
\caption{Accelerated randomized coordinate descent.}
\label{alg:rand_acc_cd}
\begin{algorithmic}[1]
  \REQUIRE{$\mu > 0$, partition $\{J_i\}_{i=1}^{m}$, positive definite $\{B_i\}_{i=1}^{m}$, Lipschitz constants $\{L_i\}_{i=1}^{m}$, $x_0 \in \R^{n}$.}
  \STATE Set $\tau = \frac{\sqrt{\mu}}{\sum_{i=1}^{m} \sqrt{L_i}}$.
  \STATE Set $H_i = S_i B_i^{-1} S_i^\T$ for $i=1,...,m$. \texttt{// $S_i$ denotes the column selector for partition $J_i$.}
  \STATE Set $p_i = \frac{\sqrt{L_i}}{\sum_{j=1}^{m} \sqrt{L_j}}$ for $i=1, ..., m$.
  \STATE Set $y_0 = z_0 = x_0$.
\FOR{$k=0, ..., T-1$}
    \STATE $i_k \gets$ random sample from $\{1, ..., m\}$ with $\Pr(i_k = i) = p_i$.
    \STATE $x_{k+1} = \frac{1}{1+\tau} y_k + \frac{\tau}{1+\tau} z_k$.
    \STATE $y_{k+1} = x_{k+1} - \frac{1}{L_{i_k}} H_{i_k} \nabla f(x_{k+1})$.
    \STATE $z_{k+1} = z_k + \tau(x_{k+1} - z_k) - \frac{\tau}{\mu p_{i_k}} H_{i_k} \nabla f(x_{k+1})$.
\ENDFOR
\STATE Return $y_T$.
\end{algorithmic}
\end{algorithm}

Theorem~\ref{thm:main_structural} is readily applied to Algorithm~\ref{alg:rand_acc_cd} to give a
convergence guarantee which matches the bound of Theorem 1 of Nesterov and Stich.
We sketch the argument below.

Algorithm~\ref{alg:rand_acc_cd} instantiates \eqref{eq:new_update_rule} with the definitions
above and particular choices $\Gamma_k = L_{i_k}$ and $\gamma_k = p_{i_k}$.
We will specify the choice of $\mu$ at a later point.
To see that this setting is valid,
we construct a discrete probability measure on $\mathcal{S}^{n \times n} \times \R_+ \times \R_+$
by setting $\omega_i = (H_i, L_i, p_i)$ and $\Pr(\omega = \omega_i) = p_i$ for $i=1, ..., m$.
Hence, in the context of Theorem~\ref{thm:main_structural},
$G = \E[ \frac{1}{\gamma} H] = \sum_{i=1}^{m} H_i = \mathrm{blkdiag}(B_1^{-1}, B_2^{-1}, ..., B_m^{-1})$.
We first verify the gradient inequality \eqref{eq:gradient_step_progress}.
For every fixed $x \in \R^{n}$, for every $i=1,...,m$ there exists a $c_i \in \R^{n}$ such that
\begin{align*}
    f(\Phi(x;\omega_i)) &= f(x) - \frac{1}{L_i}\ip{\nabla f(x)}{H_i \nabla f(x)} + \frac{1}{2L_i^2} \nabla f(x)^\T H_i \nabla^2 f(c_i) H_i \nabla f(x) \\
    &= f(x) - \frac{1}{L_i}\ip{\nabla f(x)}{H_i \nabla f(x)} \\
    &\qquad+ \frac{1}{2L_i^2} \nabla f(x)^\T S_i B_i^{-1/2} B_i^{-1/2} S_i^\T \nabla^2 f(c_i) S_i B_i^{-1/2} B_i^{-1/2} S_i^\T \nabla f(x) \\
    &\leq f(x) - \frac{1}{L_i}\ip{\nabla f(x)}{H_i \nabla f(x)} + \frac{1}{2L_i} \nabla f(x)^\T S_i B_i^{-1} S_i^\T \nabla f(x) \\
    &= f(x) - \frac{1}{2L_i} \norm{\nabla f(x)}^2_{H_i} \:.
\end{align*}
We next compute the $\nu$ constant defined in \eqref{eq:nu_parameter}.
We do this by checking the sufficient condition that $H_i G^{-1} H_i \preccurlyeq \nu H_i$ for $i=1,...,m$.
Doing so yields that $\nu=1$, since
\begin{align*}
    H_i G^{-1} H_i = S_i B_i^{-1} S_i^\T \mathrm{blkdiag}(B_1, B_2, ..., B_m) S_i B_i^{-1} S_i^\T = S_i B_i^{-1} B_i B_i^{-1} S_i^\T = S_i B_i^{-1} S_i^\T = H_i \:.
\end{align*}
To complete the argument, we set $\mu$ as the strong convexity constant
and $L$ as the Lipschitz gradient constant of $f$
with respect to the $\norm{\cdot}_{G^{-1}}$ norm.
It is straightforward to check that
\begin{align*}
    \mu = \inf_{x \in \R^{n}} \lambda_{\max}( G^{1/2} \nabla^2 f(x) G^{1/2} ) \:, \:\:
    L = \sup_{x \in \R^{n}} \lambda_{\max}( G^{1/2} \nabla^2 f(x) G^{1/2} )  \:.
\end{align*}
We now argue that $\sqrt{L} \leq \sum_{i=1}^{m} \sqrt{L_i}$.
Let $x \in \R^{n}$ achieve the supremum in the definition of $L$
(if no such $x$ exists, then let $x$ be arbitrarily close and take limits). Then,
\begin{align*}
    L &= \lambda_{\max}(G^{1/2} \nabla^2 f(x) G^{1/2}) = \lambda_{\max}( (\nabla^2 f(x))^{1/2} G (\nabla^2 f(x))^{1/2} ) \\
      &= \lambda_{\max}\left( (\nabla^2 f(x))^{1/2} \left(\sum_{i=1}^{m} S_i B_i^{-1} S_i^\T\right) (\nabla^2 f(x))^{1/2} \right) \\
      &\stackrel{(a)}{\leq} \sum_{i=1}^{m} \lambda_{\max}( (\nabla^2 f(x))^{1/2} S_i B_i^{-1} S_i^\T (\nabla^2 f(x))^{1/2} ) \\
      &\stackrel{(b)}{=} \sum_{i=1}^{m} \lambda_{\max}( S_iS_i^\T \nabla^2 f(x) S_iS_i^\T S_i B_i^{-1} S_i^\T ) \\
      &= \sum_{i=1}^{m} \lambda_{\max}( (S_i B_i^{-1} S_i^\T)^{1/2}  S_iS_i^\T \nabla^2 f(x) S_iS_i^\T (S_i B_i^{-1} S_i^\T)^{1/2} ) \\
      &\stackrel{(c)}{=} \sum_{i=1}^{m} \lambda_{\max}( S_i B_i^{-1/2} S_i^\T  S_iS_i^\T \nabla^2 f(x) S_iS_i^\T S_i B_i^{-1/2} S_i^\T ) \\
      &\stackrel{(d)}{=} \sum_{i=1}^{m} \lambda_{\max}(B_i^{-1/2} S_i^\T \nabla^2 f(x) S_i B_i^{-1/2}) \leq \sum_{i=1}^{m} L_i \:.
\end{align*}
Above, (a) follows by the convexity of the maximum eigenvalue, (b) holds since
$S_i^\T S_i = I$, (c) uses the fact that for any matrix $Q$ satisfying $Q^\T Q
= I$ and $M$ positive semi-definite, we have $(Q M Q^\T)^{1/2} = Q M^{1/2} Q^\T$, and (d) follows since
$\lambda_{\max}(S_i M S_i^\T) = \lambda_{\max}(M)$ for any $p \times p$
symmetric matrix $M$.  Using the fact that $\sqrt{a+b} \leq \sqrt{a}+\sqrt{b}$
for any non-negative $a,b$, the inequality $\sqrt{L} \leq \sum_{i=1}^{m}
\sqrt{L_i}$ immediately follows.  To conclude the proof, it remains to
calculate the requirement on $\tau$ via \eqref{eq:tau_requirements}.  Since
$\frac{\gamma_i}{\sqrt{\Gamma_i}} = \frac{p_i}{\sqrt{L_i}} =
\frac{1}{\sum_{i=1}^{m} \sqrt{L_i}}$, we have that
$\frac{\gamma_i}{\sqrt{\Gamma_i}} \leq \frac{1}{\sqrt{L}}$, and hence the
requirement is that $\tau \leq \frac{\sqrt{\mu}}{\sum_{i=1}^{m} \sqrt{L_i}}$.

\subsection{Relating Algorithm~\ref{alg:rand_acc_cd} to ACDM}
\label{sec:proofs:relating_acdm}

For completeness, we replicate the description of the ACDM algorithm from
Nesterov and Stich in Algorithm~\ref{alg:rand_acc_cd_orig}.  We make one minor
tweak in the initialization of the $A_k, B_k$ sequence which greatly simplifies
the exposition of what follows.

\begin{algorithm}[h!]
\caption{ACDM from Nesterov and Stich~\cite{nesterov16}, $\alpha=1, \beta=1/2$ case.}
\label{alg:rand_acc_cd_orig}
\begin{algorithmic}[1]
\REQUIRE{$\mu > 0$, partition $\{J_i\}_{i=1}^{m}$, positive definite $\{B_i\}_{i=1}^{m}$, Lipschitz constants $\{L_i\}_{i=1}^{m}$, $x_0 \in \R^{n}$.}
\STATE Set $H_i = S_i B_i^{-1} S_i^\T$ for $i=1,...,m$. \texttt{// $S_i$ denotes the column selector for partition $J_i$.}
\STATE Set $p_i = \frac{\sqrt{L_i}}{\sum_{j=1}^{m} \sqrt{L_j}}$ for $i=1, ..., m$.
\STATE Set $A_0 = 1, B_0 = \mu$. \texttt{// Modified from $A_0 = 0, B_0 = 1$.}
\STATE Set $S_{1/2} = \sum_{i=1}^{m} \sqrt{L_i}$.
\STATE Set $y_0 = z_0 = x_0$.
\FOR{$k=0, ..., T-1$}
    \STATE $i_k \gets$ random sample from $\{1, ..., m\}$ with $\Pr(i_k = i) = p_i$.
    \STATE $a_{k+1} \gets$ positive solution to $a_{k+1}^2 S_{1/2}^2 = (A_k + a_{k+1})(B_k + \mu a_{k+1})$.
    \STATE $A_{k+1} = A_k + a_{k+1}, B_{k+1} = B_k + \mu a_{k+1}$.
    \STATE $\alpha_k = \frac{a_{k+1}}{A_{k+1}}, \beta_k = \mu \frac{a_{k+1}}{B_{k+1}}$.
    \STATE $y_k = \frac{(1-\alpha_k)x_k + \alpha_k(1-\beta_k)z_k}{1-\alpha_k \beta_k}$.
    \STATE $x_{k+1} = y_k - \frac{1}{L_{i_k}} H_{i_k} \nabla f(y_k)$.
    \STATE $z_{k+1} = (1-\beta_k)z_k + \beta_k y_k - \frac{a_{k+1}}{B_{k+1} p_{i_k}} H_{i_k} \nabla f(y_k)$.
\ENDFOR
\STATE Return $x_T$.
\end{algorithmic}
\end{algorithm}

We first write the sequence produced by Algorithm~\ref{alg:rand_acc_cd_orig}
as
\begin{subequations}
  \begin{align}
    y_k &= \frac{(1-\alpha_k)x_k + \alpha_k(1-\beta_k)z_k}{1-\alpha_k \beta_k} \:, \label{eq:nest_orig:coupling} \\
    x_{k+1} &= y_k - \frac{1}{L_{i_k}} H_{i_k} \nabla f(y_k) \:, \label{eq:nest_orig:gradient_step} \\
    z_{k+1} - z_k &= \beta_k\left( y_k - z_k - \frac{a_{k+1}}{B_{k+1} p_{i_k} \beta_k} H_{i_k} \nabla f(y_k) \right) \:. \label{eq:nest_orig:zseq}
  \end{align}
\end{subequations}
Since $\beta_k B_{k+1} = \mu a_{k+1}$, the $z_{k+1}$ update simplifies to
\begin{align*}
  z_{k+1} - z_k &= \beta_k\left( y_k - z_k - \frac{1}{ \mu p_{i_k}} H_{i_k} \nabla f(y_k) \right) \:.
\end{align*}
A simple calculation shows that
\begin{align*}
  (1-\alpha_k \beta_k) y_k = (1-\alpha_k) x_k + \alpha_k (1-\beta_k) z_k \:,
\end{align*}
from which we conclude that
\begin{align}
  \frac{\alpha_k (1-\beta_k)}{1-\alpha_k} (y_k - z_k) = x_k - y_k \:. \label{eq:reduction_one}
\end{align}
Observe that
\begin{align*}
  A_{k+1} = \sum_{i=1}^{k+1} a_i + A_0 \:, \:\: B_{k+1} = \mu \sum_{i=1}^{k+1} a_i + B_0 \:.
\end{align*}
Hence as long as $\mu A_0 = B_0$ (which is satisfied by our modification),
we have that $\mu A_{k+1} = B_{k+1}$ for all $k \geq 0$.
With this identity, we have that $\alpha_k = \beta_k$ for all $k \geq 0$.
Therefore, \eqref{eq:reduction_one} simplifies to
\begin{align*}
  \beta_k (y_k - z_k) = x_k - y_k \:.
\end{align*}
We now calculate the value of $\beta_k$. At every iteration, we have that
\begin{align*}
  a_{k+1}^2 S^2_{1/2} = A_{k+1} B_{k+1} = \frac{1}{\mu} B_{k+1}^2 \Longrightarrow \frac{a_{k+1}}{B_{k+1}} = \frac{1}{\sqrt{\mu} S_{1/2}} \:.
\end{align*}
By the definition of $\beta_k$,
\begin{align*}
  \beta_k = \mu \frac{a_{k+1}}{B_{k+1}} = \frac{\sqrt{\mu}}{S_{1/2}} = \frac{\sqrt{\mu}}{\sum_{i=1}^{m} \sqrt{L_i}} = \tau \:.
\end{align*}
Combining these identities, we have shown that
\eqref{eq:nest_orig:coupling}, \eqref{eq:nest_orig:gradient_step},
and \eqref{eq:nest_orig:zseq} simplifies to
\begin{subequations}
  \begin{align}
    y_k &= \frac{1}{1+\tau} x_k + \frac{\tau}{1+\tau} z_k \:, \\
    x_{k+1} &= y_k - \frac{1}{L_{i_k}} H_{i_k} \nabla f(y_k) \:, \\
    z_{k+1} - z_k &= \tau\left( y_k - z_k - \frac{1}{\mu p_{i_k}} H_{i_k} \nabla f(y_k) \right) \:.
  \end{align}
\end{subequations}
This sequence directly coincides with the sequence generated by Algorithm~\ref{alg:rand_acc_cd}
after a simple relabeling.

\subsection{Accelerated Gauss-Seidel for fixed partitions from ACDM}
\label{sec:proofs:recovering_alg1}

\begin{algorithm}[h!]
\caption{Accelerated randomized block Gauss-Seidel for fixed partitions \cite{nesterov16}.}
\label{alg:rand_acc_nesterov}
\begin{algorithmic}[1]
    \REQUIRE{$A \in \R^{n \times n}$, $A \succ 0$, $b \in \R^{n}$, $x_0 \in \R^{n}$, block size $p$, $\mu_{\mathrm{part}}$ defined in \eqref{eq:mu_part}.}
\STATE Set $A_0 = 0, B_0 = 1$.
\STATE Set $\sigma = \frac{n}{p} \mu_{\mathrm{part}}$.
\STATE Set $y_0 = z_0 = x_0$.
\FOR{$k=0, ..., T-1$}
    \STATE $i_k \gets$ uniform from $\{1, 2, ..., n/p\}$.
    \STATE $S_k \gets$ column selector associated with partition $J_{i_k}$.
    \STATE $a_{k+1} \gets$ positive solution to $a_{k+1}^2 (n/p)^2 = (A_k + a_{k+1})(B_k + \sigma a_{k+1})$.
    \STATE $A_{k+1} = A_k + a_{k+1}, B_{k+1} = B_k + \sigma a_{k+1}$.
    \STATE $\alpha_k = \frac{a_{k+1}}{A_{k+1}}, \beta_k = \sigma \frac{a_{k+1}}{B_{k+1}}$.
    \STATE $y_k = \frac{(1-\alpha_k)x_k + \alpha_k(1-\beta_k)z_k}{1-\alpha_k \beta_k}$.
    \STATE $x_{k+1} = y_k - S_k(S_k^\T A S_k)^{-1} S_k^\T(A y_k - b)$.
    \STATE $z_{k+1} = (1-\beta_k)z_k + \beta_k y_k - \frac{n a_{k+1}}{p B_{k+1}} S_k(S_k^\T A S_k)^{-1} S_k^\T(A y_k - b)$.
\ENDFOR
\STATE Return $x_T$.
\end{algorithmic}
\end{algorithm}

We now describe Algorithm~\ref{alg:rand_acc_nesterov}, which is the specialization of ACDM (Algorithm~\ref{alg:rand_acc_cd_orig})
to accelerated Gauss-Seidel in the fixed partition setting.

As mentioned previously, we set the function $f(x) = \frac{1}{2} x^\T A x - x^\T b$.
Given a partition $\{J_i\}_{i=1}^{n/p}$, we let $B_i = S_i^\T A S_i$,
where $S_i \in \R^{n \times p}$ is the column selector matrix associated to the partition $J_i$.
With this setting, we have that $L_1 = L_2 = ... = L_{n/p} = 1$, and hence
we have $p_i = p/n$ for all $i$ (i.e. the sampling distribution is uniform over all partitions).
We now need to compute the strong convexity constant $\mu$.
With the simplifying assumption that the partitions are ordered,
$\mu$ is simply the strong convexity constant with respect to the
norm induced by the matrix $\mathrm{blkdiag}(B_1, B_2, ..., B_{n/p})$.
Hence, using the definition of $\mu_{\mathrm{part}}$ from \eqref{eq:mu_part}, we have that
$\mu = \frac{n}{p} \mu_{\mathrm{part}}$.
Algorithm~\ref{alg:rand_acc_nesterov} now follows from plugging our
particular choices of $f$ and the constants into Algorithm~\ref{alg:rand_acc_cd_orig}.

\section{A Result for Randomized Block Kaczmarz}
\label{sec:appendix:kaczmarz}

We now use Theorem~\ref{thm:main_structural} to derive a result similar to
Theorem~\ref{thm:rand_acc_gs} for the randomized accelerated Kaczmarz
algorithm. In this setting, we let $A \in \R^{m \times n}$, $m \geq n$ be a
matrix with full column rank, and $b \in \R^{m}$ such that $b \in
\mathcal{R}(A)$.  That is, there exists a unique $x_* \in \R^{n}$ such that $A
x_* = b$.
We note that this section generalizes the result of \cite{liu16} to the block case
(although the proof strategy is quite different).

We first describe the randomized accelerated block Kaczmarz algorithm in
Algorithm~\ref{alg:rand_acc_rk}.
Our main convergence result concerning Algorithm~\ref{alg:rand_acc_rk} is
presented in Theorem~\ref{thm:rand_acc_rk}.

\begin{algorithm}[h!]
\caption{Accelerated randomized block Kaczmarz.}
\label{alg:rand_acc_rk}
\begin{algorithmic}[1]
\REQUIRE{$A \in \R^{m \times n}$, $A$ full column rank, $b \in \mathcal{R}(A)$, sketching matrices $\{S_k\}_{k=0}^{T-1} \subseteq \R^{m \times p}$, $x_0 \in \R^{n}$, $\mu \in (0,1)$, $\nu \geq 1$.}
  \STATE Set $\tau = \sqrt{\mu/\nu}$.
  \STATE Set $y_0 = z_0 = x_0$.
\FOR{$k=0, ..., T-1$}
    \STATE $x_{k+1} = \frac{1}{1+\tau} y_k + \frac{\tau}{1+\tau} z_k$.
    \STATE $y_{k+1} = x_{k+1} - (S_k^\T A)^{\dag} S_k^\T (A x_{k+1} - b)$. \label{alg:line:localupdate}
    \STATE $z_{k+1} = z_k + \tau(x_{k+1} - z_k) - \frac{\tau}{\mu} (S_k^\T A)^{\dag} S_k^\T (A x_{k+1} - b)$.
\ENDFOR
\STATE Return $y_T$.
\end{algorithmic}
\end{algorithm}

\begin{theorem}
\label{thm:rand_acc_rk}
(Theorem~\ref{thm:rand_acc_rk_simple} restated.)
Let $A$ be an $m \times n$ matrix with full column rank, and $b \in \mathcal{R}(A)$.
Let $x_* \in \R^{n}$ denote the unique vector satisfying $A x_* = b$.
Suppose each $S_k$, $k =0, 1, 2, ...$ is an independent copy of a random sketching matrix $S \in \R^{m \times p}$.
Let $\mu = \lambda_{\min}( \E[P_{A^\T S}] )$. Suppose the distribution of $S$
satisfies $\mu > 0$.
Invoke Algorithm~\ref{alg:rand_acc_rk} with $\mu$ and $\nu$, where $\nu$ is defined as
\begin{align}
    \nu = \lambda_{\max}\left( \E\left[ (G^{-1/2} H G^{-1/2})^2 \right] \right) \:, \:\:  G = \E[H] \:, \:\: H = P_{A^\T S} \:. \label{eq:rk_variance_term}
\end{align}
Then for all $k \geq 0$ we have
\begin{align}
    \E[\twonorm{ y_k - x_* }] \leq \sqrt{2} \left(1 - \sqrt{\frac{\mu}{\nu}}\right)^{k/2} \twonorm{ x_0 - x_* } \:. \label{eq:rand_acc_rk_rate}
\end{align}
\end{theorem}
\begin{proof}
The proof is very similar to that of Theorem~\ref{thm:rand_acc_gs}, so we only sketch the main argument.
The key idea is to use the correspondence between
randomized Kaczmarz and coordinate descent (see e.g. Section 5.2 of \cite{lee13}).
To do this, we apply Theorem~\ref{thm:main_structural} to $f(x) = \frac{1}{2} \twonorm{x - x_*}^2$.
As in the proof of Theorem~\ref{thm:rand_acc_gs}, we construct
a probability measure on $\mathcal{S}^{n \times n} \times \R_+ \times \R_+$
from the given random matrix $H$ by considering the random variable $(H, 1, 1)$.
To see that the sequence
\eqref{eq:acc:coupling}, \eqref{eq:acc:gradient_step}, and \eqref{eq:acc:zseq}
induces the same update sequence as Algorithm~\ref{alg:rand_acc_rk}, the crucial step
is to notice that
\begin{align*}
    H_k \nabla f(x_{k+1}) &= P_{A^\T S_k} \nabla f(x_{k+1}) = A^\T S_k (S_k^\T AA^\T S_k)^{\dag} S_k^\T A (x_{k+1} - x_*) \\
    &= A^\T S_k (S_k^\T AA^\T S_k)^{\dag} S_k^\T (A x_{k+1} - b) = (S_k^\T A)^{\dag} S_k^\T (A x_{k+1} - b) \:.
\end{align*}
Next, the fact that $f$ is $\lambda_{\min}(\E[P_{A^\T S}])$-strongly convex
and $1$-Lipschitz with respect to the $\norm{\cdot}_{G^{-1}}$ norm, where $G = \E[P_{A^\T S}]$,
follows immediately by a nearly identical argument used in the proof of
Theorem~\ref{thm:rand_acc_gs}.
It remains to check the gradient inequality \eqref{eq:gradient_step_progress}.
%We proceed by a similar argument as in Theorem~\ref{thm:rand_acc_gs}.
Let $x \in \R^{n}$ be fixed.
Then using the fact that $f$ is quadratic, for almost every $\omega \in \Omega$,
    \begin{align*}
        f(\Phi(x;\omega)) &= f(x) - \ip{\nabla f(x)}{H(x - x_*)} + \frac{1}{2} \norm{H(x - x_*)}^2_2  \\
        &= f(x) - \ip{x-x_*}{P_{A^\T S} (x - x_*)} + \frac{1}{2} \norm{ P_{A^\T S}( x - x_*) }^2_2 \\
        &= f(x) - \frac{1}{2}\ip{x-x_*}{P_{A^\T S} (x - x_*)} \:.
    \end{align*}
Hence the gradient inequality \eqref{eq:gradient_step_progress} holds with equality.
\end{proof}

\subsection{Computing $\nu$ and $\mu$ in the setting of \cite{liu16}}
\label{sec:proofs:compute_nu_mu_liu_wright}

We first state a proposition which will be useful in our analysis of $\nu$.
\begin{proposition}
\label{prop:projector_sum_inequality}
Let $M_1, ..., M_s \subseteq \R^n$ denote subspaces of $\R^n$
such that $M_1 + ... + M_s = \R^n$.
Then we have
  \begin{align*}
    \sum_{i=1}^{s} P_{M_i} \left(\sum_{i=1}^{s} P_{M_i}\right)^{-1} P_{M_i} \preccurlyeq \sum_{i=1}^{s} P_{M_i} \:.
  \end{align*}
\end{proposition}
\begin{proof}
We will prove that for every $1 \leq i \leq s$,
\begin{align}
    P_{M_i} \left(\sum_{i=1}^{s} P_{M_i}\right)^{-1} P_{M_i} \preccurlyeq P_{M_i} \:, \label{eq:individual_ineq}
\end{align}
from which the claim immediately follows.
By Schur complements, \eqref{eq:individual_ineq} holds
iff
\begin{align}
    0 \preccurlyeq \begin{bmatrix} P_{M_i} & P_{M_i} \\ P_{M_i} & \sum_{i=1}^{s}  P_{M_i} \end{bmatrix} &= \begin{bmatrix}  P_{M_i} & P_{M_i} \\ P_{M_i} & P_{M_i} \end{bmatrix} + \begin{bmatrix} 0 & 0 \\ 0 & \sum_{j \neq i}^{s}  P_{M_j}  \end{bmatrix} \nonumber \\
    &= \begin{bmatrix} 1 & 1 \\ 1 & 1 \end{bmatrix} \otimes P_{M_i} + \begin{bmatrix} 0 & 0 \\ 0 & \sum_{j \neq i}^{s} P_{M_j} \end{bmatrix} \:.\nonumber
\end{align}
Since the eigenvalues of a Kronecker product are given by the Cartesian product of the individual eigenvalues,
\eqref{eq:individual_ineq} holds.
\end{proof}

Now we can estimate the $\nu$ and $\mu$ values.
Let $a_i \in \R^n$ denote each row of $A$, with $\norm{a_i}_2=1$ for all $i=1, ..., m$.
In this setting, $H = P_{a_i} = a_ia_i^\T$ with probability $1/m$.
Hence, $G = \E[H] = \sum_{i=1}^{m} \frac{1}{m} a_ia_i^\T = \frac{1}{m} A^\T A$.
Furthermore,
\begin{align*}
  \E[ H G^{-1} H ] &= \sum_{i=1}^{m} a_ia_i^\T  m (A^\T A)^{-1} a_ia_i^\T \frac{1}{m} \\
  &= \sum_{i=1}^{m} a_ia_i^\T (A^\T A)^{-1} a_ia_i^\T \\
  &\stackrel{(a)}{\preccurlyeq} \sum_{i=1}^{m} a_ia_i^\T = A^\T A = m G \:,
\end{align*}
where (a) follows from Proposition~\ref{prop:projector_sum_inequality}.
Hence, $\nu \neq m$.
On the other hand,
\begin{align*}
  \mu = \lambda_{\min}(\E[P_{A^\T S}]) = \lambda_{\min}(G) = \frac{1}{m} \lambda_{\min}(A^\T A) \:.
\end{align*}

%% file: proofs_random_coordinates.tex
\section{Proofs for Random Coordinate Sampling (Section~\ref{sec:results:rand_coords})}

Our primary goal in this section is to provide a proof of
Lemma~\ref{lemma:nu_upper_bound}. Along the way, we prove a few other
results which are of independent interest.
We first provide a proof of the lower bound claim in Lemma~\ref{lemma:nu_upper_bound}.
\begin{proposition}
\label{prop:nu_lower_bound}
Let $A$ be an $n \times n$ matrix and let
$S \in \R^{n \times p}$ be a random matrix.
Put $G = \E[ P_{A^{1/2} S} ]$ and suppose that $G$ is positive definite.
Let $\nu > 0$ be any positive number such that
\begin{align}
    \E[ P_{A^{1/2} S} G^{-1} P_{A^{1/2} S} ] \preccurlyeq \nu G \:, \:\: G = \E[ P_{A^{1/2} S} ] \:. \label{eq:nu_inequality_again}
\end{align}
Then $\nu \geq n/p$.
\end{proposition}
\begin{proof}
Since trace commutes with expectation and
respects the positive semi-definite ordering, taking trace of
both sides of \eqref{eq:nu_inequality_again} yields that
\begin{align*}
    n &= \Tr( G G^{-1} ) =
    \Tr( \E[ P_{A^{1/2} S} G^{-1}] ) =
    \E[\Tr(  P_{A^{1/2} S} G^{-1}  )] =
    \E[\Tr(  P_{A^{1/2} S} G^{-1} P_{A^{1/2} S}  )] \\
    &= \Tr(\E[ P_{A^{1/2} S} G^{-1} P_{A^{1/2} S} ] ) \,\,\,\eqrefstackrel{eq:nu_inequality_again}{\leq} \,\,\, \nu \Tr(\E[ P_{A^{1/2} S}]) \\
    &= \nu \E[ \Tr(P_{A^{1/2} S}) ] = \nu \E[ \rank(A^{1/2} S) ] \leq \nu p \:.
\end{align*}
\end{proof}
Next, the upper bound relies on the following lemma,
which generalizes Lemma~2 of \cite{qu16}.
\begin{lemma}
\label{lemma:semidef_jensen}
Let $M$ be a random matrix. We have that
\begin{align}
    \E[P_{M}] \succcurlyeq \E[M] (\E[M^\T M])^{\dag} \E[M^\T] \:.
\end{align}
\end{lemma}
\begin{proof}
Our proof follows the strategy in the proof of Theorem 3.2 from \cite{zhang05}.
First, write $P_B = B(B^\T B)^{\dag}B^\T$. Since $\Range(B^\T) = \Range(B^\T B)$,
we have by generalized Schur complements (see e.g. Theorem 1.20 from \cite{zhang05})
and the fact that expectation preserves the semi-definite order,
    \begin{align*}
        \begin{bmatrix} B^\T B & B^\T \\ B & P_B \end{bmatrix} \succcurlyeq 0 \Longrightarrow \begin{bmatrix} \E[B^\T B] & \E[B^\T] \\ \E[B] & \E[P_B] \end{bmatrix} \succcurlyeq 0 \:.
    \end{align*}
To finish the proof, we need to argue that $\Range(\E[B^\T]) \subseteq \Range(\E[B^\T B])$,
which would allow us to apply the generalized Schur complement again to the right hand side.
Fix a $z \in \Range(\E[B^\T])$; we can write $z = \E[B^\T]y$ for some $y$.
Now let $q \in \Kern(\E[B^\T B])$. We have that $\E[B^\T B] q = 0$, which implies $0 = q^\T \E[B^\T B] q = \E[\twonorm{B q}^2]$.
Therefore, $B q = 0$ a.s.  But this means that $z^\T q = \E[y^\T B q] = 0$. Hence, $z \in \Kern(\E [B^\T B])^{\perp} = \Range(\E[B^\T B])$.
Now applying the generalized Schur complement one more time yields the claim.
\end{proof}
We are now in a position to prove the upper bound of Lemma~\ref{lemma:nu_upper_bound}.
We apply Lemma~\ref{lemma:semidef_jensen} to $M=A^{1/2} SS^\T A^{1/2}$ to conclude,
using the fact that $\mathcal{R}(M) = \mathcal{R}(MM^\T)$,
that
\begin{align}
    \E[P_{A^{1/2} S}] &= \E[P_{A^{1/2} SS^\T A^{1/2}}] \succcurlyeq \E[A^{1/2} SS^\T A^{1/2}] ( \E[A^{1/2} SS^\T A SS^\T A^{1/2}])^{\dag} \E[A^{1/2} SS^\T A^{1/2}] \label{eq:lower_bound_gs} \:.
\end{align}
Elementary calculations now yield that for any fixed symmetric matrix $A \in \R^{n \times n}$,
\begin{align}
    \E[SS^\T] = \frac{p}{n} I, \qquad \E[SS^\T A SS^\T] = \frac{p}{n}\left(\frac{p-1}{n-1} A + \left(1-\frac{p-1}{n-1}\right) \diag(A)\right) \:. \label{eq:simple_calcs}
\end{align}
Hence plugging \eqref{eq:simple_calcs} into \eqref{eq:lower_bound_gs},
\begin{align}
    \E[P_{A^{1/2} S}] &\succcurlyeq \frac{p}{n} \left( \frac{p-1}{n-1} I + \left(1-\frac{p-1}{n-1}\right) A^{-1/2} \diag(A) A^{-1/2} \right)^{-1} \:. \label{eq:semidef_lower_bound_proj}
\end{align}
We note that the lower bound \eqref{eq:mu_rand_lower_bound} for $\mu_{\mathrm{rand}}$ presented
in Section~\ref{sec:background} follows immediately from \eqref{eq:semidef_lower_bound_proj}.

We next manipulate \eqref{eq:gs_variance_term} in order to use \eqref{eq:semidef_lower_bound_proj}.
Recall that $G = \E[ H ]$ and $H = S(S^\T A S)^{\dag} S^\T$.
From \eqref{eq:nu_param_equiv}, we have
\begin{align*}
    \lambda_{\max}\left( \E\left[ (G^{-1/2} H G^{-1/2})^2 \right] \right) \leq \nu
    \Longleftrightarrow  \E\left[ H G^{-1} H \right] \preccurlyeq \nu G \:.
\end{align*}
Next, a simple computation yields
\begin{align*}
    \E[H G^{-1} H] = \E[S(S^\T A S)^{-1} S^\T G^{-1} S(S^\T A S)^{-1} S^\T] = A^{-1/2} \E[P_{A^{1/2} S} ( \E[P_{A^{1/2} S}] )^{-1} P_{A^{1/2} S}]  A^{-1/2} \:.
\end{align*}
Again, since conjugation by $A^{1/2}$ preserves semi-definite ordering,
we have that
\begin{align*}
    \E[ H G^{-1} H ] \preccurlyeq \nu G \Longleftrightarrow \E[P_{A^{1/2} S} ( \E[P_{A^{1/2} S}] )^{-1} P_{A^{1/2} S}] \preccurlyeq \nu \E [P_{A^{1/2} S}]  \:.
\end{align*}
Using the fact that for positive definite matrices $X, Y$ we have
$X \preccurlyeq Y$ iff $Y^{-1} \preccurlyeq X^{-1}$,
\eqref{eq:semidef_lower_bound_proj} is equivalent to
\begin{align*}
    (\E[P_{A^{1/2} S}])^{-1} \preceq \frac{n}{p} \left(\frac{p-1}{n-1} I + \left(1 - \frac{p-1}{n-1}\right) A^{-1/2} \diag(A) A^{-1/2} \right) \:.
\end{align*}
Conjugating both sides by $P_{A^{1/2} S}$ and taking expectations,
\begin{align}
    \E [P_{A^{1/2} S}(\E[P_{A^{1/2} S}])^{-1}P_{A^{1/2} S}] &\preceq \frac{n}{p}\left( \frac{p-1}{n-1} \E[P_{A^{1/2} S}] + \left(1 - \frac{p-1}{n-1}\right) \E[ P_{A^{1/2} S}  A^{-1/2} \diag(A) A^{-1/2} P_{A^{1/2} S}] \right) \:. \label{eq:intermediate_bound}
\end{align}
Next, letting $J \subseteq 2^{[n]}$ denote the index set associated to $S$, for every $S$ we have
\begin{align*}
    &P_{A^{1/2} S}  A^{-1/2}\diag(A) A^{-1/2} P_{A^{1/2} S} \\
    &\qquad= A^{1/2} S (S^\T A S)^{-1} S^\T A^{1/2} A^{-1/2} \diag(A) A^{-1/2} A^{1/2} S (S^\T A S)^{-1} S^\T A^{1/2} \\
    &\qquad= A^{1/2} S (S^\T A S)^{-1/2}  (S^\T A S)^{-1/2} (S^\T \diag(A) S) (S^\T A S)^{-1/2}  (S^\T A S)^{-1/2} S^\T A^{1/2} \\
    &\qquad\preccurlyeq \lambda_{\max}( (S^\T \diag(A) S) (S^\T A S)^{-1} ) A^{1/2} S (S^\T A S)^{-1} S^\T A^{1/2} \\
    &\qquad\preccurlyeq \frac{\max_{i \in J} A_{ii}}{\lambda_{\min}(A_{J})} P_{A^{1/2} S} \\
    &\qquad\preccurlyeq \max_{J \in 2^{[n]} : \abs{J}=p} \kappa_{\mathrm{eff},J}(A) P_{A^{1/2} S} \:.
\end{align*}
Plugging this calculation back into \eqref{eq:intermediate_bound} yields the desired upper bound of
Lemma~\ref{lemma:nu_upper_bound}.